\documentclass[a4paper,11pt]{amsart}
\usepackage{amsfonts,amssymb,amsmath,amsthm,abstract,color}
\usepackage[mathscr]{euscript}
\usepackage[ps,all,arc,rotate]{xy}
\usepackage[lmargin=1in,rmargin=1in,tmargin=1in,bmargin=1in]{geometry}
\usepackage{fancyhdr}
\usepackage{pb-diagram}
\usepackage[shortlabels]{enumitem}
\usepackage{hyperref}

\newtheorem{prop}{Proposition}[section]
\newtheorem{lemma}[prop]{Lemma}
\newtheorem{thm}[prop]{Theorem}
\newtheorem{cor}[prop]{Corollary}

\theoremstyle{definition}
\newtheorem{defn}[prop]{Definition}
\newtheorem{question}[prop]{Question}

\newtheorem{rmk}[prop]{Remark}
\newtheorem{ex}[prop]{Example}

\newcommand{\Hecke}{\cH ecke}

\DeclareMathOperator{\Chow}{Chow}

\DeclareMathOperator{\rk}{rk}        
 \DeclareMathOperator{\Sym}{Sym}

\DeclareMathOperator{\id}{id}

\newcommand{\ra}{\rightarrow}

\newcommand{\rap}{\stackrel{+}{\ra}}
\newcommand{\Gy}{\mathrm{Gy}}

\DeclareMathOperator*{\hocolim}{hocolim}

\DeclareMathOperator{\Jac}{Jac}\DeclareMathOperator{\Sch}{Sch} 
\DeclareMathOperator{\Hom}{Hom}
\DeclareMathOperator{\Spec}{Spec}

\DeclareMathOperator{\codim}{codim}

\DeclareMathOperator{\eff}{{eff}}
\DeclareMathOperator{\Nis}{Nis}

\DeclareMathOperator{\Aut}{Aut}

\def\cC{\mathcal C}\def\cD{\mathcal D}
\def\cE{\mathcal E}\def\cF{\mathcal F}\def\cG{\mathcal G}\def\cH{\mathcal H}
\def\cI{\mathcal I}\def\cJ{\mathcal J}
\def\cO{\mathcal O}
\def\cR{\mathcal R}\def\cT{\mathcal T}
\def\cU{\mathcal U}
\def\cY{\mathcal Y}

\def\AA{\mathbb A}\def\CC{\mathbb C}
\def\FF{\mathbb F}\def\GG{\mathbb G}

\def\NN{\mathbb N}\def\PP{\mathbb P}
\def\QQ{\mathbb Q}\def\RR{\mathbb R}
\def\VV{\mathbb V}
\def\ZZ{\mathbb Z}

\def\fg{\mathfrak g}

\def\fX{\mathfrak X}

 \def\GL{\mathrm{GL}} 

\def\DM{\mathrm{DM}}     
\def\CH{\mathrm{CH}}

\makeatletter
\def\@tocline#1#2#3#4#5#6#7{\relax
  \ifnum #1>\c@tocdepth % then omit
  \else
    \par \addpenalty\@secpenalty\addvspace{#2}%
    \begingroup \hyphenpenalty\@M
    \@ifempty{#4}{%
      \@tempdima\csname r@tocindent\number#1\endcsname\relax
    }{%
      \@tempdima#4\relax
    }%
    \parindent\z@ \leftskip#3\relax \advance\leftskip\@tempdima\relax
    \rightskip\@pnumwidth plus4em \parfillskip-\@pnumwidth
    #5\leavevmode\hskip-\@tempdima
      \ifcase #1
       \or\or \hskip 1em \or \hskip 2em \else \hskip 3em \fi%
      #6\nobreak\relax
    \hfill\hbox to\@pnumwidth{\@tocpagenum{#7}}\par% <---- \dotfill -> \hfill
    \nobreak
    \endgroup
  \fi}
\makeatother

\makeatletter
\newsavebox{\@brx}
\newcommand{\llangle}[1][]{\savebox{\@brx}{\(\m@th{#1\langle}\)}%
  \mathopen{\copy\@brx\kern-0.5\wd\@brx\usebox{\@brx}}}
\newcommand{\rrangle}[1][]{\savebox{\@brx}{\(\m@th{#1\rangle}\)}%
  \mathclose{\copy\@brx\kern-0.5\wd\@brx\usebox{\@brx}}}
\makeatother

\newcommand{\Bun}{\mathcal{B}un_{n,d}}
\newcommand{\Ch}{\mathcal{C}h_{\underline{n},\underline{d}}}
\newcommand{\Chgs}{\Ch^{\mathrm{gen-surj}}}
\newcommand{\Chtau}{\Ch^{\alpha,\tau}}

\newcommand{\Chss}{\Ch^{\alpha,ss}}
\newcommand{\Chs}{\Ch^{\alpha,s}}

\DeclareMathOperator{\ch}{Ch}
\DeclareMathOperator{\h}{H}
\DeclareMathOperator{\m}{M}

\newcommand{\mCh}{\ch^{\alpha,ss}_{\underline{n},\underline{d}}}
\newcommand{\mChs}{\ch^{\alpha,s}_{\underline{n},\underline{d}}}
\newcommand{\mH}{\h^{ss}_{n,d}}
\newcommand{\mdR}{\m^{\mathrm{dR}}_{n,d}}
\newcommand{\mHdg}{\m^{\mathrm{Hdg}}_{n,d}}
\newcommand{\mHs}{\h^{s}_{n,d}}

\title[On the Voevodsky motive of the moduli space of Higgs bundles on a curve]{On the Voevodsky motive of the moduli space of Higgs bundles on a curve}

\author{Victoria Hoskins and Simon Pepin Lehalleur}

\thanks{V.H. is supported by the DFG Excellence Initiative at the Freie Universit\"{a}t Berlin and by the SPP 1786.}

\begin{document}

\maketitle

\begin{abstract}
We study the motive of the moduli space of semistable Higgs bundles of coprime rank and degree on a smooth projective curve $C$ over a field $k$ under the assumption that $C$ has a rational point. We show this motive is contained in the thick tensor subcategory of Voevodsky's triangulated category of motives with rational coefficients generated by the motive of $C$. Moreover, over a field of characteristic zero, we prove a motivic non-abelian Hodge correspondence: the integral motives of the Higgs and de Rham moduli spaces are isomorphic. 
\end{abstract}

\tableofcontents

\section{Introduction}

\subsection{Moduli of Higgs bundles and their cohomological invariants}

Let $C$ be a smooth projective geometrically connected genus $g$ curve over a field $k$. A Higgs bundle over $C$ is a vector bundle $E$ together with a Higgs field, which is an $\mathcal{O}_{C}$-linear map $E\to E\otimes \omega_{C}$. There is a notion of (semi)stability analogous to the classical notion for vector bundles and a construction via geometric invariant theory of the moduli space $\mH$ of semistable Higgs bundles of rank $n$ and degree $d$ over $C$. We assume that $n$ and $d$ are coprime; this implies that semistability coincides with stability, and that $\mH$ is a smooth quasi-projective variety of dimension $2n^2(g-1) +2$.

Hitchin's original motivation for introducing Higgs bundles and their moduli in \cite{hitchin} came from mathematical physics, but these spaces now play a central role in many subfields of geometry. Most notably, over the complex numbers, moduli spaces of Higgs bundles are (non-compact) hyperk\"{a}hler manifolds which are isomorphic as real analytic manifolds to moduli spaces of representations of the fundamental group of $C$ and moduli spaces of holomorphic connections via Simpson's non-abelian Hodge correspondence \cite{simpson_NAHT}. 

For a long time, the cohomology of $\mH$ was quite mysterious. One recent breakthrough was the precise conjectural formulae for the Betti numbers over the complex numbers by Hausel and Rodriguez-Villegas \cite{hauselRV} (predicted via point-counting arguments for character varieties over finite fields); the conjecture was recently proved by Schiffmann \cite{schiffmann}, Mozgovoy-Schiffmann \cite{ms} and Mellit \cite{mellit} by counting absolutely indecomposable vector bundles over finite fields, in the spirit of Kac's theory for quiver representations, and using Hall algebra techniques.

Our paper follows a different geometric strategy, which can be traced back to the original paper of Hitchin \cite{hitchin}. There, he used a scaling $\GG_{m}$-action on the Higgs field to compute the Betti numbers of $\h_{2,d}^{ss}(\mathbb{C})$, and Gothen \cite{gothen} extended this approach to rank $3$. This scaling action was later studied by Simpson \cite{simpson_NAHT} in higher ranks. The components of the $\GG_{m}$-fixed loci are moduli spaces of chains of vector bundles on $C$ which are semistable with respect to a certain stability parameter, and the cohomology of $\mH$ can be described in terms of the cohomology of these moduli spaces of chains by using the classical techniques of Bia{\l}ynicki-Birula \cite{BB}. Since these moduli spaces of semistable chains are smooth projective varieties, this shows that the cohomology of $\mH$ is pure (this purity was observed by Hausel and Thaddeus \cite[Theorem 6.2]{HT} and Markman \cite{markman}). Moreover, the problem of describing the cohomology of $\mH$ reduces to the problem of describing the cohomology of moduli spaces of chains which are semistable with respect to a specific stability parameter. 

The classes of moduli spaces of Higgs bundles and moduli spaces of chains in the Grothendieck ring of varieties were studied by Garc\'{i}a-Prada, Heinloth and Schmitt \cite{GPHS,GPH,Heinloth_LaumonBDay}. One key geometric idea is to vary the chain stability parameter and use wall-crossing in terms of unions of Harder--Narasimhan strata to inductively write the classes of moduli stacks of semistable chains for certain stability parameters using classes of simpler stacks. More precisely, it suffices to compute the classes of certain moduli stacks of generically surjective chains of constant rank, for which there are explicit formulas in terms of the classes of symmetric powers of the curve and classes of stacks of vector bundles over the curve. The latter were determined by Behrend and Dhillon \cite{BD}. This gives a recursive algorithm to compute the class of $\mH$ in the Grothendieck ring of varieties \cite{GPH} and it follows that the class of $\mH$ can be expressed in terms of classes of symmetric powers of $C$, the Jacobian of $C$ and powers of the Lefschetz class.

\subsection{Our results}

In this paper, we study the motive $M(\mH)$ in Voevodsky's triangulated category $\DM^{\eff}(k,R)$ of (effective) mixed motives over $k$ with coefficients in a ring $R$ such that the exponential characteristic of $k$ is invertible in $R$ under the assumption that $C(k) \neq \emptyset$. By construction $M(\mH)$ lies in the subcategory $\DM^{\eff}_{c}(k,R)$ of compact objects of $\DM^{\eff}(k,R)$.

Our first main result adapts the geometric ideas in \cite{GPHS,GPH,Heinloth_LaumonBDay} to Voevodsky's triangulated category $\DM^{\eff}(k,R)$ of mixed motives over $k$.

\begin{thm}\label{main_thm_intro}
Assume that $C(k) \neq \emptyset$ and that $R$ is a $\QQ$-algebra. Then the motive $M(\mH)$ lies in the thick tensor subcategory $\langle M(C)\rangle^{\otimes}$ of $\DM^{\eff}_{c}(k,R)$ generated by $M(C)$. More precisely, $M(\mH)$ can be written as a direct factor of the motive of a large enough power of $C$.
\end{thm}

In fact, as $C(k) \neq \emptyset$, the stack $\cH^{ss}_{n,d}$ of semistable Higgs bundles is a trivial $\GG_m$-gerbe over $\mH$ and the ideas in the proof of Theorem \ref{main_thm_intro} provides a more precise description of the motive of $\cH^{ss}_{n,d}$ as fitting in a explicit sequence of distinguished triangles built from known motives (Corollary \ref{cor motive Higgs stack}), which we hope will lead to precise computations in small ranks. The promised description as a direct factor is unfortunately completely inexplicit and relies on a general observation about pure motives and weight structures; see Lemma \ref{lemma:weight-structure}.

Our second main result, over a field of characteristic zero, compares the integral motives of $\mH$ and the de Rham moduli space $\mdR$ appearing in the non-abelian Hodge correspondence (see $\S$\ref{sec motivic_NAH} for the precise definition). This motivic counterpart to the non-abelian Hodge correspondence is relatively easy to prove (and has some antecedents for cohomology in the literature) but does not seem to have been observed before, even at the level of Chow groups.

\begin{thm}\label{thm:motivic-hodge-intro}
Assume that $k$ is a field of characteristic zero and $C(k)\neq \emptyset$. For any commutative ring $R$, there is a canonical isomorphism
  \[
M(\mH)\simeq  M(\mdR)
  \]
  in $\DM^{\eff}(k,R)$ which also induces an isomorphism of Chow rings
  \[
\CH^{*}(\mH,R)\simeq \CH^{*}(\mdR,R).
    \]
\end{thm}

The proof of Theorem \ref{main_thm_intro} and of its refined form (Corollary \ref{cor motive Higgs stack}) is a priori more complicated than in the context of the Grothendieck ring of varieties, as the cut and paste relations $[X] = [Z] + [X \setminus Z]$ for a closed subvariety $Z \subset X$ play a key role in the proof of \cite{GPH}, whereas in $\DM^{\eff}(k,R)$ one only has an associated localisation distinguished triangle (and Gysin triangle for smooth pairs) which does not split in general.

However, there are some circumstances in which these triangles split: in particular, the Gysin triangles associated to a Bia{\l}ynicki-Birula decomposition of a smooth projective variety with a $\GG_m$-action split and one obtains motivic decompositions \cite{Brosnan, Choudhury_Skowera, Karpenko} and in fact, in the appendix to this paper, we show this is also true for smooth quasi-projective varieties with a so-called semi-projective $\GG_m$-action (see Theorem \ref{thm_bb_mot}). This provides the first step of our argument: since the scaling action on the moduli space of Higgs bundles $\mH$ is semi-projective, the motive of $\mH$ can be expressed in terms of Tate twists of motives of certain moduli spaces of semistable chains, which are smooth and projective. From this, we already deduce the purity of $M(\mH)$.

It is at this point in the logical development, by combining motivic Bia{\l}ynicki-Birula decompositions with the geometry of the Deligne-Simpson moduli space of $\lambda$-connections, that we prove Theorem \ref{thm:motivic-hodge-intro} (see Theorem \ref{thm:motivic-hodge}); in fact, it is obtained as a corollary of a general fact about equivariant specialisations of semi-projective $\GG_m$-actions (Theorem \ref{thm:semiproj-spe}).

The motives of the moduli spaces of semistable chains appearing in the Bia{\l}ynicki-Birula decomposition of $\mH$ can be expressed in terms of the motives of the corresponding moduli stacks of semistable chains when $C(k) \neq \emptyset$, since this implies these stacks are trivial $\GG_m$-gerbes over their coarse moduli space.  Therefore, it remains to describe the motives of the corresponding moduli stacks of semistable chains. We follow the geometric ideas in \cite{GPHS,GPH,Heinloth_LaumonBDay}, which involves a wall-crossing argument together with a Harder--Narasimhan (HN) recursion.

The space of stability parameters for chains has a wall and chamber structure such that (semi)stability is constant in chambers and as one crosses a wall, the stacks of semistable chains for the wall parameter is a union of the stacks of semistable chains and finitely many stacks of chains of fixed HN type for the stability parameters on either side of wall. We use a path in the space of stability parameters constructed in \cite{GPHS} starting from (a small perturbation of) the Higgs stability parameter $\alpha_H$ which ends in a chamber where either the semistable locus is empty or is contained in a moduli stack of generically surjective chains of constant rank. We obtain a diagram of distinguished triangles relating the motives of the stacks of $\alpha_H$-semistable chains we are interested in with the motives of stacks of generically surjective chains of constant rank and higher HN strata for various stability parameters along the path.

The motives of moduli stacks of chains of fixed (non-trivial) HN type can be described inductively using the fact that the map taking a chain to its the associated graded for the HN filtration is a Zariski locally trivial affine space fibration over a product of moduli stacks of semistable chains of smaller ranks.

The last step of the computation, and the one which requires the most additional work compared to \cite{GPHS}, is to lift the formula for the classes of stacks of generically surjective chains of constant rank in \cite[Lemma 4.9]{GPHS} to $\DM^{\eff}(k,R)$ and this is where we are forced to assume that $R$ is a $\QQ$-algebra. These stacks turn out to be iterated moduli stacks of Hecke correspondences over moduli stacks of vector bundles over $C$. In \cite{HPL,HPL_formula} we prove a formula for the rational motive of the stack $\Bun$ of rank $n$ degree $d$ vector bundles on $C$ under the assumption that $C(k)\neq \emptyset$; this formula involves motives of symmetric powers of the curve, the Jacobian of the curve and Tate twists (see \cite[Theorem 1.1]{HPL_formula}).  
The arguments in \cite{HPL_formula} involve calculating the motive of varieties of Hecke correspondences for a family of vector bundles over $C$ parameterised by a smooth variety $T$ (see \cite[Theorem 3.8]{HPL_formula}). We were inspired by a more sheaf-theoretic argument of Heinloth \cite[Proof of Proposition 11]{Heinloth_LaumonBDay}, which gives the rational cohomology of stacks of Hecke correspondences over any base using an argument based on ideas of Laumon \cite{laumon} involving the cohomology of small maps which are generically principal bundles. Heinloth's proof uses perverse sheaves and cannot be applied in $\DM^{\eff}(k,R)$, but we gave a more geometric form of the argument which can be made to work using the six operations formalism for \'etale motives. It remains to extend this formula to the case where we replace the smooth variety $T$ by the smooth stack $\Bun$. Since in \cite[Theorem 3.2]{HPL} we prove that $\Bun$ is a so-called exhaustive stack (see \cite[Definition 2.15]{HPL}), it suffices to prove the formula extends to smooth exhaustive stacks, which is what we do in \S \ref{sec small maps}. More precisely, we obtain the following result, which completes the proof of Theorem \ref{main_thm_intro}.

\begin{thm}\label{thm_hecke}
Assume that $R$ is a $\QQ$-algebra. Let $\cE$ be a family of rank $n$ vector bundles over $C$ parametrised by a smooth exhaustive algebraic stack $\cT$. Then the stack $\Hecke^{l}_{\cE/\cT}$ of length $l$ Hecke modifications of $\cE$ (i.e. subsheaves $\cF \subset \cE$ whose quotient is a family of length $l$ torsion sheaves) is smooth and exhaustive, and we have an isomorphism
\[ M(\Hecke^l_{\cE/\cT}) \cong M(\cT) \otimes M(\Sym^l(C \times \PP^{n-1}))\]
in $\DM^{\eff}(k,R)$.
\end{thm}

\noindent
\textbf{Notation and conventions.}

\noindent
Let $\cD$ be a tensor triangulated category admitting small direct sums and such that tensor products preserve small direct sums (for instance $\cD=\DM^{\eff}(k,R)$). Let $\cG$ be a set of objects in $\cD$. We denote by $\langle \cG \rangle$ (resp. $\llangle \cG \rrangle$) the smallest thick (resp. localising) subcategory of $\cD$ containing $\cG$; that is, the smallest triangulated subcategory of $\cD$ containing $\cG$ and furthermore stable by taking direct factors (resp. small direct sums). We also denote by $\langle \cG \rangle^{\otimes}$ the smallest thick (resp. localising) tensor subcategory of $\cD$ containing $\cG$; that is, the smallest triangulated subcategory of $\cD$ containing $\cG$ and furthermore stable by taking tensor products and direct factors (resp. tensor products and small direct sums).

It is easy to show that the subcategory $\langle \cG \rangle$ (resp. $\llangle \cG \rrangle$) admits a more concrete iterative description as a countable (resp. transfinite) union of full subcategories $(\langle \cG \rangle^{n})_{n\geq 0}$ (resp. $(\llangle \cG \rrangle^{\alpha})_{\alpha\text{ ordinal}}$), individually not triangulated in general, with
\begin{itemize}
\item $\langle \cG \rangle^{0}=\langle \cG\rangle^{0}$ the full subcategory on the set $\cup_{k\in\ZZ}\cG[k]$ of shifts of objects in $\cG$.
\item for all $n>0$ (resp. for all $\alpha>0$), $\langle \cG \rangle^{n}$ (resp. $\llangle \cG \rrangle^{\alpha}$) the full subcategory of objects which are either extensions or direct factors (resp. extensions or small direct sums) of objects in $\cup_{m<n} \langle \cG  \rangle^{m}$ (resp. $\cup_{\beta<\alpha} \llangle \cG \rrangle^{\beta}$).
\end{itemize}

\noindent \textbf{Acknowledgements.} We would like to thank Alexander Schmitt for suggesting in the first place to study the Voevodsky motive of the moduli space of Higgs bundles along these lines. We thank H\'el\`ene Esnault, Michael Groechenig, Jochen Heinloth, Francesco Sala and Carlos Simpson for conversations and exchanges around this project.

\section{Moduli of Higgs bundles and moduli of chains}

Throughout this section, we fix a smooth projective geometrically connected genus $g$ curve $C$ over a field $k$ and coprime integers $n\in \NN$ and $d\in \ZZ$. Let us introduce the main object of this paper, the moduli space $\mH$ of semistable Higgs bundles over $C$ of rank $n$ and degree $d$. 

\subsection{Moduli of Higgs bundles}

A Higgs bundle over $C$ is a pair $(E,\Phi)$ consisting of a vector bundle $E$ and a homomorphism $\Phi: E \ra E \otimes \omega_C$ called the Higgs field. The numerical invariants of the Higgs bundle are given by the rank $\rk(E)$ and degree $\deg(E)$ of the vector bundle. 

\begin{defn}
The slope of a Higgs bundle $(E,\Phi)$ is defined by $\mu(E):=\deg(E)/\rk(E)$. The Higgs bundle $(E,\Phi)$ is (semi)stable if for all Higgs subbundles $E'\subset E$ (that is, a vector subbundle $E' \subset E$ such that $\Phi(E') \subset E' \otimes \omega_C$), we have
\[ \mu(E') \: (\leq) \: \mu(E), \]
where $(\leq)$ denotes $\leq$ for semistability and $<$ for stability. We say $(E,\Phi)$ is geometrically (semi)stable if its pullback to $C_K:= C \times_k K$ is (semi)stable for all field extensions $K/k$.
\end{defn}

If we consider Higgs bundles of coprime rank and degree, then the notions of stability and semistability coincide. Every Higgs bundle has a unique \lq Harder--Narasimhan' filtration by Higgs subbundles whose successive quotients are semistable of strictly decreasing slopes. The uniqueness of the Harder--Narasimhan filtration can be used to show that the notions of semistability and geometric semistability coincide over any field. Over an algebraically closed field, the notions of stability and geometric stability coincide; however, over a non-algebraically closed field these notions can differ (see \cite[1.3.9]{HL}). For us all these notions will coincide, as $n$ and $d$ will be assumed to be coprime.

There is a moduli space $\mH$ of semistable Higgs bundles over $C$ with fixed invariants $n$ and $d$, which is a quasi-projective variety that can be constructed via geometric invariant theory \cite{simpson_IHES_I}. It contains an open subvariety $\mHs$, the moduli space of geometrically stable Higgs bundles, which is a smooth variety of dimension $2n^2(g-1) +2$ whose geometric points correspond to isomorphism classes $[E,\Phi]$ of stable Higgs bundles. Moreover, every (semi)stable vector bundle is a (semi)stable Higgs bundle with any Higgs field and the deformation theory of vector bundles implies that the cotangent bundle to the moduli space of semistable vector bundles is contained in the moduli space of semistable Higgs bundles.  

\subsection{The scaling action on the moduli space of Higgs bundles}\label{sec Gm action}

In this section we will exploit a natural $\GG_m$-action on the moduli space $\mH$ of semistable Higgs bundles over $C$ whose fixed loci are moduli spaces of semistable chains. This action was first used by Hitchin \cite{hitchin} to compute the Betti numbers of $\mH$ when $n = 2$ and was later used by Simpson \cite{simpson_NAHT} for higher ranks. The $\GG_m$-action is defined by scaling the Higgs field: for $t \in \GG_m$ and $[E,\Phi] \in \mH$, let
\[ t \cdot [E,\Phi] := [E, t \cdot \Phi].\]
The fixed loci and the flow under this $\GG_m$-action are described by the following result.

\begin{prop}[Hitchin, Simpson]\label{prop Higgs BB decomp}
The above $\GG_m$-action on $\mH$ is semi-projective (in the sense of Definition \ref{def semiproj Gm}) and thus there is a Bia{\l}ynicki-Birula decomposition
\[ \mH = \bigsqcup_{i \in I} H_i^+, \]
where $H_i$ are the connected components of $(\mH)^{\GG_m}$ and $H_i^+$ is the locally closed subvariety of $\mH$ consisting of points $x$ such that $\lim_{t \ra 0} t \cdot x \in H_i$. Moreover, the fixed components $H_i$ are smooth projective varieties and the natural retraction $H_i^+ \ra H_i$ is a Zariski locally trivial affine bundle. The strata $H_i^+$ are smooth locally closed subvarieties of $\mH$ of dimension $\frac{1}{2} \dim \mH$.
\end{prop}
\begin{proof}
The fact that the $\GG_m$-action is semi-projective is due to Hitchin and Simpson and then by work of Bia{\l}ynicki-Birula \cite{BB} (see Theorem \ref{thm_bb_dec}), there exists a decomposition with the above description; for more details on the proof, see \cite[Section 9]{HT}.
\end{proof}

In fact, the fixed loci components also have a moduli theoretic description due to Hitchin and Simpson. If an isomorphism class $[E,\Phi]$ of a semistable Higgs bundle is fixed by this action, then either $\Phi = 0$ or we have $\GG_m \subset \Aut(E)$ which induces a weight decomposition $E = \oplus_i E_i$ such that $\Phi(E_i) \subset E_{i+1} \otimes \omega_C$, as $\GG_m$ acts on $\Phi$ with weight $1$. If $i_0$ denotes the minimum weight for which $E_{i_0}$ is non-zero, then we obtain a chain of vector bundle homomorphisms
\[ E_{i_0} \ra E_{i_0 +1} \otimes \omega_C \ra E_{i_0+2} \otimes \omega_C^{\otimes 2} \ra \cdots \]
which terminates after finitely many homomorphisms, as there are only finitely many weights appearing in the decomposition $E= \oplus E_i$. If we write $F_i:= E_{i_0 + i} \otimes \omega_C^{\otimes i}$, then this gives us a chain
\[ F_0 \ra F_1 \ra F_2 \ra \cdots \ra F_r \]
for some $r \in \NN$. Since $\omega_{C}$ is a line bundle, the $E_{i}$'s are uniquely determined by the $F_{i}$'s. The case $r= 0$ corresponds to vanishing Higgs field $\Phi = 0$; thus one fixed component is the moduli space of semistable rank $n$ degree $d$ vector bundles on $C$.  The other components of the fixed locus will be moduli spaces of semistable chains as we explain below.

\subsection{Moduli of chains}\label{sec moduli chains}

The fixed loci for the above scaling action on $\mH$ are moduli spaces of chains which are semistable with respect to a certain stability parameter. We give some basic properties of moduli of chains and explain their relationship with Higgs bundles.

A chain of length $r$ over $C$ is a collection of vector bundles $(F_i)_{i=0, \dots, r}$ and homomorphisms $(\phi_i : F_{i-1} \ra F_{i})_{i=1,\cdots r}$ between these vector bundles, which we write as 
\[ F_\bullet = (F_0 \stackrel{\phi_1}{\ra} F_{1} \ra \cdots \ra F_{r-1} \stackrel{\phi_r}{\ra} F_r).\]
The invariants of this chain are the tuples of ranks and degrees $\underline{\rk} (F_\bullet):= (\rk F_i)_{i=0,\dots , r}$ and $\underline{\deg} (F_\bullet):= (\deg F_i)_{i=0,\dots , r}$. There are natural notions of homomorphisms of chains and the category of chains over $C$ is an abelian category.

If we fix tuples $\underline{n}$ and $\underline{d}$ of ranks and degrees, then there is an algebraic stack $\Ch$ of chains with these invariants, which is locally of finite type \cite[\S 4.1]{GPHS}. There are notions of (semi)stability for chains depending on a stability parameter $\alpha$ and also natural notions of Harder--Narasimhan filtrations with respect to such a stability parameter. 

\begin{defn}\label{def slopes chains}
Let $\alpha= (\alpha_i)_{i=0,\dots, r} \in \RR^{r+1}$ be a tuple of real numbers. 
\begin{enumerate}
\item We define the $\alpha$-slope of a chain $F_\bullet$ by
\[ \mu_{\alpha}(F_\bullet) = \frac{\sum_{i=0}^r (\deg F_i + \alpha_i \rk(F_i))}{\sum_{i=0}^r \rk F_i }. \]
Note that this only depends on the numerical invariants $(\rk(F_\bullet),\deg(F_\bullet))$ of the chain. A chain $F_\bullet$ is $\alpha$-(semi)stable if for all proper subchains $F'_\bullet \subset F_\bullet$ we have
\[ \mu_{\alpha}(F'_\bullet) (\leq ) \mu_{\alpha}(F_\bullet),\]
where $(\leq)$ denotes $\leq$ for semistability and $<$ for stability. We say $F_\bullet$ is geometrically $\alpha$-(semi)stable if its pullback to $C_K$ is $\alpha$-(semi)stable for all field extensions $K/k$.
\item The stability parameter $\alpha$ is critical for given numerical invariants $\underline{n}$ and $\underline{d}$ if there exists $\underline{0} <\underline{n}' < \underline{n}$ (that is, $0 \leq n'_{i}\leq n_{i}$ for all $i$ with at least one strict inequality) and $\underline{d}'$ such that $\mu_{\alpha}(\underline{n}',\underline{d}') = \mu_{\alpha}(\underline{n},\underline{d})$. Otherwise, we say $\alpha$ is non-critical for these invariants.
\item Every chain has a unique Harder--Narasimhan (HN) filtration with respect to $\alpha$
\[ 0 = F^{(0)}_\bullet \subset F^{(1)}_\bullet \subset \cdots \subset F^{(l)}_\bullet = F_\bullet \]
such that $F^i_\bullet:= F^{(i)}_\bullet / F^{(i+1)}_\bullet$ are $\alpha$-semistable with strictly decreasing $\alpha$-slopes \cite[Lemma 4.2]{GPHS}. The $\alpha$-HN type of $F_\bullet$ records the numerical invariants of the subquotients and we write this as a tuple $ (\underline{\rk}(F^{j}_\bullet),\underline{\deg}(F^{j}_\bullet))_{j=1,\dots, l}$. 
We let $\Chss$ denote the substack of $\alpha$-semistable chains and let $\Chtau$ denote the substack of chains of $\alpha$-HN type $\tau$. The substack $\Chss$ is open in $\Ch$ while each $\Chtau$ is locally closed.
\end{enumerate}
\end{defn}

By definition, if $\alpha$ is non-critical for $\underline{n}$ and $\underline{d}$, the notions of $\alpha$-semistability and $\alpha$-stability for chains with these invariants coincide. Similarly to the discussion for Higgs bundles, if $\alpha$ is non-critical for $\underline{n}$ and $\underline{d}$, then all these notions of (semi)stability coincide.

There are moduli spaces $\mCh$ of $\alpha$-semistable chains over $C$ with fixed invariants $\underline{n}$ and $\underline{d}$ which are projective varieties that can be constructed as geometric invariant theory quotients (see \cite{schmitt_moduli}). Furthermore, the deformation theory of chains is described in \cite[Section 3]{ACGPS}.

Since geometrically $\alpha$-stable chains have automorphism groups isomorphic to the multiplicative group $\GG_m$, the stack of geometrically $\alpha$-stable chains $\Chs$ is a $\GG_m$-gerbe over its coarse moduli space $\mChs$, the moduli space of geometrically $\alpha$-stable chains over $C$. Moreover, if we assume $C(k) \neq \emptyset$, then for non-critical values of $\alpha$ for $\underline{n}$ and $\underline{d}$, the $\GG_m$-gerbe $\Chs \ra \mChs$ is trivial (for example, this can be proved using \cite[Lemma 3.10]{heinloth_lectures}). %see also \cite{GPHS} page 39.

Consider the following cones of stability parameters
\begin{equation*}
\Delta_r:=\{ \alpha \in \RR^{r+1} : \alpha_i - \alpha_{i+1} \geq 2g -2 \} \quad \text{and} \quad \Delta_r^\circ :=\{ \alpha \in \RR^{r+1} : \alpha_i - \alpha_{i+1} > 2g -2 \}.
\end{equation*}
If $\alpha \in \Delta_r$ and $\alpha$ is non-critical for the invariants $\underline{n}$ and $\underline{d}$, then the moduli spaces $\mCh = \mChs$ (and also the corresponding moduli stacks) are smooth by \cite[Theorem 3.8 vi)]{ACGPS}. For $\alpha \in \Delta_r^\circ$, the moduli stack $\Chss$ of $\alpha$-semistable chains is smooth by \cite[Proposition 3.5 ii)]{ACGPS} (see also \cite[Lemma 4.6]{GPHS} and \cite[Section 2.2]{heinloth_int_form}). Moreover, the HN-strata for stability parameters in the cone $\Delta_r^\circ$ admit the following description.

\begin{prop}\label{prop HN strata geom descr}
\cite[Lemma 4.6 and Proposition 4.8]{GPHS}
Let $\alpha \in \Delta_r^\circ$. For an $\alpha$-HN type $\tau = (\underline{n}^j,\underline{d}^j)_{j=1,\dots, l}$, the morphism given by taking the associated graded for the $\alpha$-HN filtration
\[ \mathrm{gr} : \Chtau \longrightarrow \prod_{j=1}^l \mathcal{C}h_{\underline{n}^j,\underline{d}^j}^{\alpha,ss} \]
is an affine space fibration. 
Thus the stack $\Chtau$ of chains with $\alpha$-HN type $\tau$ is also smooth.
\end{prop}

To state the relationship between the fixed point set of the $\GG_m$-action on $\mH$ and moduli spaces of chains, we introduce a \emph{Higgs stability parameter} for length $r$ chains
\[ \alpha_H:= (r(2g-2), \dots ,2g-2,0 ) \]
which lies on the boundary of the above cone $\Delta_{r}$.

\begin{prop}\label{prop ss chain and Higgs}
Any length $r$ chain $F_\bullet$ determines a Higgs bundle $(E = \oplus_{i=0}^r F_i \otimes \omega_C^{\otimes - i}, \Phi)$, where the Higgs field $\Phi : E \ra E \otimes \omega_C$ is determined by the chain homomorphisms $\phi_i : F_{i-1} \ra F_i$. Furthermore, the associated Higgs bundle $(E,\Phi)$ is (semi)stable if and only if the chain $F_\bullet$ is $\alpha_H$-(semi)stable. 
\end{prop}
\begin{proof} 
  This is essentially due to Hitchin \cite{hitchin} and Simpson \cite{simpson}: one verifies that $E' = \oplus E_i'$ is a Higgs subbundle of $(E = \oplus E_i, \Phi)$ if and only if $F'_\bullet$ is a subchain of $F'_\bullet$, where $E_i':= F_i' \otimes \omega_C^{\otimes - i}$. Moreover, one has
  \[  \mu_{\alpha_H}(F'_\bullet)=  \frac{\sum_{i=0}^r \deg F'_i + (r-i)(2g-2) \rk(F'_i)}{\sum_{i=0}^r \rk F'_i }  = \frac{\sum_{i=0}^r \deg E'_i}{\sum_{i=0}^r \rk E'_i } +r(2g-2) = \mu(E') +r(2g-2)\] 
and so (semi)stability of $(E,\Phi)$ corresponds to $\alpha_H$-(semi)stability of $F_\bullet$.
\end{proof}

\begin{cor}\label{cor fixed locus modular}
The connected components of the fixed point set of the $\GG_m$-action on $\mH$ are moduli spaces of $\alpha_H$-semistable chains for numerical invariants $\underline{n}$ and $\underline{d}$ for which $\alpha_H$ is non-critical (and thus the notions of semistability and stability with respect to $\alpha_H$ coincide for these numerical invariants). In particular, the moduli spaces of $\alpha_H$-semistable chains appearing as fixed components are smooth projective varieties.
\end{cor}
\begin{proof}
The fact that $\alpha_H$ is non-critical for the numerical invariants $\underline{n}$ and $\underline{d}$ arising in this decomposition is a consequence of $n$ and $d$ being coprime, as $\mu_{\alpha_H} (\underline{n}', \underline{d}') = \mu_{\alpha_H} (\underline{n}, \underline{d})$ if and only if $\mu(\sum_i n_i', \sum d'_i -i(2g-2) ) = \mu(n,d)$ by the proof of Proposition \ref{prop ss chain and Higgs}. Since $\alpha_H$ is non-critical, one can find a small perturbation to a non-critical stability parameter $\widetilde{\alpha}_H \in \Delta^\circ_r$ such that $\alpha_H$-semistability coincides with $\widetilde{\alpha}_H$-semistability; then by \cite[Theorem 2]{heinloth_int_form} the moduli stack of $\widetilde{\alpha}_H$-semistable chains with these invariants is connected and so it follows that the moduli space of $\alpha_H$-semistable chains with these invariants is also connected, hence a connected component of the $\GG_{m}$-fixed point locus.
\end{proof}

\subsection{Variation of stability and Harder--Narasimhan stratification results}

By varying the parameter $\alpha \in \RR^{r+1}$, one obtains different notions of (semi)stability and correspondingly different moduli spaces of semistable chains which are related by birational transformations. One can subdivide the space of stability parameters $\RR^{n+1}$ into locally closed subsets, known as a wall and chamber decomposition, such that the notion of $\alpha$-(semi)stability is constant within each chamber and changes as one crosses a wall. In fact, the walls are hyperplanes which are in bijections with invariants $(\underline{0} <\underline{n}' < \underline{n},\underline{d}')$ which witness the criticality of a critical stability parameter in the sense of Definition \ref{def slopes chains}. This wall-crossing picture for chains was studied in \cite[Section 4]{ACGPS} and was described from the point of view of stacks in \cite[Section 3]{GPH}.

Let us recall the description of the stacky wall-crossing given in \cite[Proposition 2]{GPH}: let $\alpha_0$ be a critical stability parameter for invariants $\underline{n}$ and $\underline{d}$, for $\delta \in \RR^{r+1}$ consider the family of stability parameters $\alpha_t = \alpha_0 + t \delta$ in a neighbourhood of $t = 0 \in \RR$. Then for $0 <\epsilon < \!< 1$, we have that $\alpha_\epsilon$-(semi)stability and $\alpha_{-\epsilon}$-semistability are independent of $\epsilon$, and if we write $\alpha_{\pm}:= \alpha_{\pm \epsilon}$, then we have
\begin{equation}\label{wall-crossing}
\Ch^{\alpha_0,ss} = \Ch^{\alpha_{\pm},ss} \sqcup \bigsqcup_{\tau \in I_{\pm}} \Ch^{\alpha_{\pm},\tau}, 
\end{equation}  
where $I_{\pm}$ are finite sets of $\alpha_{\pm}$-HN types.

For the wall-crossing arguments we employ later, we will need to introduce the stack of generically surjective chains.

\begin{defn}
Let $\Chgs$ denote the substack of $\Ch$ consisting of chains 
\[ F_\bullet = (F_0 \stackrel{\phi_1}{\ra} F_{1} \ra \cdots \ra F_{r-1} \stackrel{\phi_r}{\ra} F_r)\]
such that all the homomorphisms $\phi_i$ are generically surjective. 
\end{defn}

\begin{rmk}
For $\underline{n} = (n, \dots, n)$ constant, $\Chgs$ is smooth and connected by \cite[Lemma 4.9]{GPHS} (see also \cite[Theorem 3.8 v)]{ACGPS}). In this constant rank case, as we are working over a curve, the stack of generically surjective chains coincides with the stack of injective chains (see \cite{Heinloth_LaumonBDay}).
\end{rmk}

\begin{prop}\label{prop WC}
Let $\alpha \in \Delta_r$ and let $\underline{n}$ and $\underline{d}$ be invariants for chains of length $r$. Then there is a ray $(\alpha_t)_{t \geq 0} \in \Delta_r$ starting at $\alpha_0 = \alpha$ with the following properties.
\begin{enumerate}
\item \cite[Lemma 6]{GPH} If $n_i \neq n_j$ for some $i \neq j$, then for $t > \!> 0$ there are no $\alpha_t$-semistable chains with these invariants; that is, $\Ch^{\alpha_t,ss}  = \emptyset$ for $t > \!> 0$.
\item \cite[Corollary 6.10]{GPHS}  If $n_i = n_j$ for all $i, j$, then for $t > \!> 0$ we have 
\[ \Ch^{\alpha_t,ss} \subset \Chgs \]
and moreover $\Chgs$ is an (infinite) union of $\alpha_t$-HN strata.
\end{enumerate}
If moreover $\alpha \in \Delta_r^\circ$, then the ray  $(\alpha_t)_{t \geq 0} $ can be chosen to remain in $\Delta_r^\circ$.
\end{prop}

\begin{rmk}\label{rmk pert path}
As observed in the proof of \cite[Proposition 2.6]{heinloth_int_form}, the above paths can be perturbed such that every critical value along the perturbed path lies on a single wall and the path is linear in a neighbourhood of each critical value, as the wall and chamber decomposition is a locally finite partition of $\Delta_r^\circ$ by \cite[$\S$2.4]{ACGPS}.
\end{rmk}

\section{Voevodsky's category of effective motives and motives of stacks}

\subsection{Motives of schemes}

In this section, let us briefly recall some basic properties about Voevodsky's category $\DM^{\eff}(k,R):=\DM^{\Nis,\eff}(k,R)$ of effective (Nisnevich) motives over $k$ with coefficients in a ring $R$. For the remainder of the paper, we fix such a ring $R$ and we always assume that the exponential characteristic of $k$ is invertible in $R$ (this is necessary for the more subtle properties of $\DM(k,R)$, such as the existence of the weight structure used in Section \ref{subsec purity}). In places, in particular for our main result (Theorem \ref{main_thm_intro}), we need the stronger assumption that $R=\QQ$, which we always point out explicitly.

The category $\DM^{\eff}(k,R)$ is a $R$-linear tensor triangulated category, which was originally constructed in \cite{VSF} and its deeper properties were established under the hypothesis that $k$ is perfect and satisfies resolution of singularities. These properties were extended to the case where $k$ is perfect and the exponential characteristic of $k$ is invertible in $R$ by Kelly in \cite{kelly}, using Gabber's refinement of de Jong's results on alterations.

For a separated scheme $X$ of finite type over $k$, one can associate a motive $M(X)\in \DM^{\eff}(k,R)$ which is covariantly functorial in $X$ and behaves like a homology theory. The unit for the monoidal structure is $M(\Spec k):=R\{0\}$, and there are Tate motives $R\{n\}:=R(n)[2n] \in \DM(k,R)$ for all $n\in\NN$. For any motive $M$ and $n \in \NN$, we write $M\{n\}:=M\otimes R\{n\}$. 

In $\DM^{\eff}(k,R)$, there are K\"{u}nneth isomorphisms, $\AA^1$-homotopy invariance, Gysin distinguished triangles, projective bundle formulae and Poincar\'{e} duality isomorphisms, as well as realisation functors (Betti, de Rham, $\ell$-adic,$\cdots$) and descriptions of Chow groups with coefficients in $R$ as homomorphism groups in $\DM^{\eff}(k,R)$. There is also a well-behaved subcategory $\DM^{\eff}_{c}(k,R)\subset \DM^{\eff}(k,R)$ which can be described equivalently as the subcategory of compact objects in the triangulated sense or as the thick triangulated subcategory generated by $M(X)$ for $X$ smooth.

By Voevodsky's cancellation theorem \cite{Voevodsky-canc} (together with a result of Suslin to cover the case where $k$ is imperfect \cite{Suslin-imperfect}), the category $\DM^{\eff}(k,R)$ embeds as a full subcategory of the larger triangulated category of (non-effective) motives $\DM(k,R)$. For an overview of properties of $\DM(k,R)$, which by the previous embedding also covers the main properties of $\DM^{\eff}(k,R)$ alluded to in the previous paragraph, we refer the reader to the summary in \cite[$\S$2]{HPL}.

\subsection{Motives of smooth exhaustive stacks} \label{exhaustive}

There are several approaches to defining motives of general algebraic stacks when working with rational coefficients. For a comparison of different approaches, see \cite[Appendix A]{HPL}, and for a more general approach which partially describes a six operation formalism on stacks via the natural $\infty$-category structures on rational motives, see \cite{RS}. A key point in the story is that Voevodsky motives with rational coefficients satisfy smooth cohomological descent, which ensures that the resulting motives do not depend on auxiliary choices and relate as expected to motives of finite type schemes. However, it is not clear how to construct motives of algebraic stacks with coefficients in a more general ring $R$.

In this paper, as in \cite{HPL}, we can work with a more restrictive class of algebraic stacks, namely smooth exhaustive stacks. Informally, an algebraic stack $\fX$ is exhaustive if it can be well approximated by a sequence of schemes which occur as open substacks of vector bundles over increasingly large open substacks of $\fX$ (see \cite[Definition 2.15]{HPL} for the precise definition, which we will not need in this paper). In particular, quotient stacks $[X/G]$ for a $G$-linearised action of an affine algebraic group $G$ on a smooth quasi-projective variety $X$ are smooth and exhaustive, see \cite[Lemma 2.16]{HPL}. Almost all of the stacks we consider in this paper, and all of those for which we consider an associated motive, are exhaustive, see Proposition \ref{prop gen surj} and Lemma \ref{lemma HN strata motive}.

For such smooth exhaustive stacks and for any coefficient ring $R$, we can attach a motive in $\DM^{\eff}(k,R)$ by adapting ideas of Totaro and Morel-Voevodsky (see \cite[Definition 2.17]{HPL}; note that in loc.cit. we work in $\DM(k,R)$ but the definition works as well and compatibly in the full subcategory $\DM^{\eff}(k,R)$). In \cite[Appendix A]{HPL}, we checked that this is compatible with the more general definition alluded to above when $R$ is a $\QQ$-algebra. We do not need to go into the definition, but we will need the following property which follows immediately from the definition.

\begin{prop}\label{prop:exhaustive}
  Let $\cT$ be a smooth exhaustive algebraic stack. Then there exists a diagram of finite type separated $k$-schemes
  \[
U_{0}\to U_{1}\to U_{2}\to\ldots
\]
which lives over $\cT$ and such that we have
\[
M(\cT)\simeq \hocolim_{n\in\NN} M(U_{n})
\]
in $\DM^{\eff}(k,R)$. Let $f:\cT'\to \cT$ a flat finite type separated representable morphism, and $U'_{n}:=U_{n}\times_{\cT}\cT'$ for $n\in \NN$ (so that $U'_{n}$ is also finite type and separated). Then $\cT'$ is exhaustive, and we have
\[
M(\cT')\simeq \hocolim_{n\in\NN} M(U'_{n})
\]
in $\DM^{\eff}(k,R)$.
\end{prop}

In \cite{HPL}, we established a few of the expected properties of motives of smooth exhaustive stacks, in particular we prove that the product of exhaustive stacks is exhaustive and establish K\"unneth isomorphisms, $\AA^{1}$-homotopy invariance and Gysin triangles in \cite[Proposition 2.27]{HPL}. Note that the results are stated in the larger category $\DM(k,R)$ but also hold in the full subcategory $\DM^{\eff}(k,R)$ by the Voevodsky-Suslin embedding.

\section{Motivic non-abelian Hodge correspondence}\label{sec motivic_NAH}

One of the most interesting features of the moduli space $\mH$ is its role in the non-abelian Hodge correspondence of Corlette and Simpson, which relates $\mH$ to a moduli space $\mdR$ of (logarithmic) flat connections on $C$ over the complex numbers \cite{simpson_NAHT}. Over a field $k$ of characteristic zero, there is still a geometric relationship between $\mH$ and $\mdR$, instantiated by Deligne's moduli space $\mHdg$ of (logarithmic) $\lambda$-connections. In this section, which is independent of the rest of the paper, we combine this with Appendix \ref{sec:motiv-equiv-semipr} to compare the motives of $\mH$, $\mdR$ and $\mHdg$. 

In this section, we assume that $k$ is a field of characteristic zero and, as in the rest of the paper, we assume that $n$ and $d$ are coprime and that $C(k)\neq \emptyset$. We fix $x \in C(k)$ and consider logarithmic connections with poles at $x$ of fixed residue, whose definition we recall below.

In \cite{simpson_IHES_I}, Simpson defines the notion of a sheaf $\Lambda$ of rings of differential operators over a smooth projective variety $X/S$ generalising the usual ring of differential operators $\cD_{X/S}$ and constructs moduli spaces of $\Lambda$-modules (i.e. sheaves of left $\Lambda$-modules which are coherent as $\cO_{X}$-modules) which are semistable (in the usual sense of verifying an inequality of slopes for all $\Lambda$-submodules). For special choices of $\Lambda$, one obtains moduli spaces of coherent sheaves, Higgs bundles, flat connections and more generally logarithmic connections and their degenerations given by $\lambda$-connections.

\begin{defn}
Let $\Lambda^{\mathrm{dR},\log x}$ denote the (split, almost polynomial) sheaf of differential operators over $C$ associated to the sheaf of logarithmic differentials $\Omega_C(\log x)$ and the ordinary differential (see \cite[page 87 and Theorem 2.11]{simpson_IHES_I}); then a $\Lambda^{\mathrm{dR},\log(x)}$-module is a coherent sheaf $E$ on $C$ with logarithmic connection $\nabla : E \ra E \otimes \Omega_C(\log x)$ with poles at $x$ satisfying the usual Leibniz condition (as we are over a curve, the integrability condition holds trivially). Following page 86 of \textit{loc.\ cit.}, we can construct a deformation to the associated graded $\mathrm{Gr}(\Lambda^{\mathrm{dR},\log x}) = \Sym^*(\Omega_C(\log x)^\vee)$ which is a (split, almost polynomial) sheaf of split differential operators over $C \times \AA^1$ denoted $\Lambda^{\mathrm{dR},\log x,R}$. For $\lambda \in k$, let $i_\lambda : C \times \{ \lambda \} \hookrightarrow C \times \AA^1$ denote the inclusion and let $\Lambda^{\mathrm{dR},\log x,\lambda}:= i_\lambda^*\Lambda^{\mathrm{dR},\log x,R}$; then a $\Lambda^{\mathrm{dR},\log x,\lambda}$-module is a logarithmic $\lambda$-connection on $C$ with poles at $x$. 
\end{defn}

A logarithmic $\lambda$-connection $(E,\nabla)$ admits a residue at $x$, which is an endomorphism of $E_{x}$. Fixing the value of the residue determines a closed condition in moduli. We will be interested in the case when the residue is $-\lambda\frac{d}{n} \mathrm{Id}$. 

When $\lambda$ is non-zero, it is possible to rescale a (logarithmic) $\lambda$-connection (with pole at $x$ of residue $-\lambda\frac{d}{n} \mathrm{Id}$) into an ordinary (logarithmic) connection using the natural scaling action of $\GG_m$ on $\AA^1$, while when $\lambda=0$, the residue being zero implies that we do not have a non-trivial pole and we obtain an ordinary Higgs bundle on $C$.

A special case of Simpson's general construction in \cite{simpson_IHES_I} yields a quasi-projective coarse moduli space $\mHdg$ of semistable rank $n$ degree $d$ logarithmic $\lambda$-connections with poles at $x$ of residue $-\lambda\frac{d}{n} \mathrm{Id}$ for varying $\lambda\in k$, together with a flat morphism $f:\mHdg\to \mathbb{A}^{1}$ such that $f^{-1}(0)\simeq \mH$ is the moduli space of semistable Higgs bundles, $f^{-1}(1)\simeq \mH = \mdR$ is the moduli space of semistable logarithmic connections with pole at $x$ of residue $-\frac{d}{n} \mathrm{Id}$ and $f^{-1}(\GG_m)\simeq \mdR\times \GG_m$ (by the rescaling argument above). Since $n$ and $d$ are coprime, the moduli space $\mHdg$ and morphism $f:\mHdg\to \mathbb{A}^{1}$ are both smooth. Simpson shows that the scaling action on $\mH$ extends to $\mHdg$, see \cite{simpson_IHES_I,simpson_NAHT}. Moreover, he shows that the variety $\mHdg$ together with this $\GG_m$-action is still semi-projective in the sense of Definition \ref{def semiproj Gm} (see \cite[Lemma 16]{simpson_NAHT} for the case of $\lambda$-connections of degree $0$).

It is known that, when $k=\CC$, the family $f$ is topologically trivial, so that the fibre inclusions $\mH\hookrightarrow \mHdg$ and $\mdR\hookrightarrow \mHdg$ induce isomorphisms on singular cohomology; see \cite[Proposition 15]{simpson_NAHT} and \cite[Lemma 6.1]{HT}. In the Grothendieck ring of varieties, the motivic classes of the corresponding stacks (in degree $0$, not in the coprime case) in characteristic zero are proven to be equal in \cite[Theorem 1.2.1]{FSS}.

If we believe in the conservativity conjecture of realisations with rational coefficients, the topological triviality suggests that at least the motives with rational coefficients are the same. In fact, we prove that this holds integrally.

\begin{thm}\label{thm:motivic-hodge}
Let $k$ be a field of characteristic $0$, $C/k$ be a smooth projective geometrically connected curve with $C(k)\neq \emptyset$ and let $n$ and $d$ be coprime integers. Let $R$ be a ring such that the exponential characteristic of $k$ is invertible in $R$. Then the fibre inclusions in the Deligne-Simpson family induce isomorphisms
  \[
M(\mH)\simeq M(\mHdg)\simeq M(\mdR)
  \]
  in $\DM^{\eff}(k,R)$. 
\end{thm}
\begin{proof}
By the results on the geometry of the Deligne-Simpson family recalled above, we can apply Theorem \ref{thm:semiproj-spe} and the result follows.
\end{proof}

As in Corollary \ref{cor:semiproj-inv}, it follows that the Chow rings with $R$-coefficients (resp. the $\ell$-adic cohomology, etc.) of $\mH$, $\mHdg$ and $\mdR$ are canonically isomorphic.

\section{Hecke correspondences and the stack of generically surjective chains}\label{sec small maps}

\subsection{Motives of stacks of Hecke correspondences}

In \cite[\S 3]{HPL_formula}, we proved a formula for the motives of schemes of Hecke correspondences associated to a family of vector bundles over the curve $C$ parametrised by a base scheme. In this section, we generalise this formula to the situation where the base is a smooth exhaustive algebraic stack. This situation is sufficient for our needs, and it has the advantage that the proof is then a simple application of the result for schemes. It is likely that the formula holds for more general base algebraic stacks, but this seems to require more delicate arguments. This is one of the reasons we chose to restrict to the formalism of exhaustive stacks in this paper.

For a family $\cE$ of vector bundles on $C$ parametrised by an algebraic stack $\cT$, we write $\rk(\cE) = n$ and $\deg(\cE) = d$ if the fibrewise rank and degree of this family are $n$ and $d$ respectively.

\begin{defn}
For  $l \in \NN$ and a family $\cE$ of rank $n$ degree $d$ vector bundles over $C$ parametrised by an algebraic stack $\cT$ (considered as a category fibered in groupoids over $\Sch/k$), we define a category fibered in groupoids $\Hecke^l_{\cE/\cT}$ over $\Sch/k$ as follows. For every $S\in\Sch/k$, the objects are defined as
\[ \Hecke^l_{\cE/\cT} (S):= \left\{g\in \cT(S),\phi : \cF \hookrightarrow (g \times \mathrm{id}_C)^*\cE: \begin{array}{c} \cF \ra S \times C \text{ vector bundle}  \\  \rk(\cF) = n, \deg(\cF)=d-l, \rk(\phi)=n  \end{array} \right\}. \]
Given a morphism $f:S'\to S$ in $\Sch/k$, a morphism over $f$ from $(g',\phi')\in \Hecke^l_{\cE/\cT} (S')$ to $(g,\phi)\in \Hecke^l_{\cE/\cT} (S)$ is a morphism $\alpha:g'\to g\circ f$ in $\cT(S')$ and an isomorphism $(f\times \id_{C})^{*}\cF\stackrel{\sim}{\to} \cF'$ which fits into the diagram
\[
  \xymatrix{
    (f\times \id)^{*}\cF \ar[d]_{\sim}\ \ar@{^{(}->}[r] & (f\times\id_{C})^{*}(g\times\id_{C})^{*}\cE \ar[d]_{\alpha^{*}}^{\sim} \\
    \cF'\ \ar@{^{(}->}[r] & (g'\times \id_{C})^{*} \cE.
  }
  \]
Note that the left vertical isomorphism is in fact determined by $\alpha$ because of the injectivity of the horizontal maps. We refer to $\Hecke^l_{\cE/\cT}$ as the stack of length $l$ Hecke correspondences of $\cE$.
\end{defn}

By construction, $\Hecke^l_{\cE/\cT}$ comes together with a morphism $\Hecke^l_{\cE/\cT}\to \cT$, and we have the following base change property.

\begin{lemma}\label{lemma:pullback_hecke}
  Let $f:\cT'\to \cT$ be a morphism of algebraic stacks. Let $l\in \NN$ and $\cE$ a family of rank $n$ degree $d$ vector bundles on $C$ parametrised by $\cT$. Then there is a natural equivalence of categories fibered in groupoids
  \[
\Hecke^l_{(f\times \id_{C})^*\cE/\cT'} \simeq \Hecke^l_{\cE/\cT}\times_\cT \cT'.
\]
\end{lemma}

\begin{lemma}\label{lemma:repr_hecke}
Fix $l \in \NN$ and a family $\cE$ of rank $n$ degree $d$ vector bundles over $C$ parametrised by an algebraic stack $\cT$. Then the morphism $\Hecke^{l}_{\cE/\cT}\to \cT$ is relatively representable, smooth and projective. In particular $\Hecke^{l}_{\cE/\cT}$  is an algebraic stack.
\end{lemma}  
\begin{proof}
  This follows from Lemma \ref{lemma:pullback_hecke} and the fact that $\Hecke^l_{\cE/T}$ is representable by a smooth projective (Quot) scheme over $T$ when $T$ is a scheme, as discussed in \cite[\S 3]{HPL_formula}.
\end{proof}
  
We can now prove Theorem \ref{thm_hecke}.

\begin{proof}[Proof of Theorem \ref{thm_hecke}] By Proposition \ref{prop:exhaustive} and Lemma \ref{lemma:repr_hecke}, the stack $\Hecke^{l}_{\cE/\cT}$ is smooth and exhaustive, and moreover there exists a diagram
\[ 
U_{0}\to U_{1}\to U_{2}\to\ldots
\]
of finite separated $k$-schemes over $\cT$ such that, if we write $\pi_{n}:U_{n}\to \cT$ for the structure morphisms and define $\widetilde{U}_{n}:=U_{n}\times_{\cT} \Hecke^{l}_{\cE/\cT}\simeq \Hecke^{l}_{\pi_{n}^{*}\cE/U_{n}}$, we have isomorphisms
\begin{eqnarray*}
  M(\cT) & \simeq & \hocolim_{n\in\NN} M(U_{n}), \text{ and} \\
  M(\Hecke^{l}_{\cE/\cT}) & \simeq & \hocolim_{n\in\NN} M(\widetilde{U}_{n}).
\end{eqnarray*}

By \cite[Theorem 3.8]{HPL_formula}, the $\NN$-indexed system $\{M(\Hecke^{l}_{\pi_{n}^{*}\cE/U_{n}})\}_{n\in\NN}$ is isomorphic to the tensor product of $\{M(U_{n})\}$ with the constant system $M(\Sym^{l}(C\times \PP^{n-1})$. The result then follows by passing to the homotopy colimit.
\end{proof}

\subsection{Motive of the stack of generically surjective chains}

We compute the rational motive of the stack of generically surjective chains between sheaves with constant ranks.

\begin{prop}\label{prop gen surj}
Let $\underline{n}=(n_0, \dots, n_r)$ and $\underline{d} = (d_0, \dots , d_r)$ be such that $n_i = n_{i+1}$ and $l_i:= d_i - d_{i-1}$ is non-negative for all $i$. Then the stack $\Chgs$ of generically surjective chains is an iterated Hecke correspondence stack over the moduli stack $\mathcal{B}un_{n_r,d_r}$ of vector bundles of rank $n_{r}$ and degree $d_{r}$ on $C$. It is thus a smooth exhaustive stack, and its motive in $\DM^{\eff}(k,R)$ is given by
\[
  \begin{split}
  M(\Chgs) & \simeq  M(\mathcal{B}un_{n_r,d_r}) \otimes \bigotimes_{i=1}^r M(\Sym^{l_i}(C \times \PP^{n-1})) \\
  & \simeq M(\Jac(C)) \otimes M(B\GG_m) \otimes \bigotimes_{i=1}^{n_{r}-1} Z(C, \QQ\{i\})\otimes \bigotimes_{i=1}^r M(\Sym^{l_i}(C \times \PP^{n-1}))
\end{split}  
\]
and we have $M(\Chgs)\in \llangle M(C)\rrangle^{\otimes}$.
\end{prop}
\begin{proof}  
The description of $\Chgs$ as an iterated Hecke correspondence stack over $\mathcal{B}un_{n_r,d_r}$ is contained in \cite[Lemma 4.9]{GPHS}. Let us briefly give the details here in the case of a tuple of constant ranks, in which case generically surjective homomorphisms of vector bundles of the same rank are in fact injective as we are over a curve. We write $\underline{n}_{\geq i} :=(n_i, \dots , n_r)$ and $\underline{d}_{\geq i}:=(d_i, \dots , d_r)$ and $\mathcal{C}h_{\geq i}^{\mathrm{gen-surj}} :=\mathcal{C}h_{\underline{n}_{\geq i},\underline{d}_{\geq i}}^{\mathrm{gen-surj}}$ and write $\cU_{\geq i}^i \ra  \cdots \ra \cU_{\geq i}^r$ for the universal chain over $\mathcal{C}h_{\geq i}^{\mathrm{gen-surj}} \times C$. Then there are natural forgetful maps
\[ \mathcal{C}h_{\geq i-1}^{\mathrm{gen-surj}} \ra \mathcal{C}h_{\geq i}^{\mathrm{gen-surj}}\]
which we claim are Hecke modification stacks. More precisely, we claim that 
\[ \mathcal{C}h_{\geq i-1}^{\mathrm{gen-surj}} = \Hecke^{l_{i}}(\cU_{\geq i}^i/\mathcal{C}h_{\geq i}^{\mathrm{gen-surj}}), \]
since a generically surjective chain $F_{i} \ra \cdots \ra F_r$ with invariants $\underline{n}_{\geq i}$ and $\underline{d}_{\geq i}$ together with a length $l_{i}$ Hecke modification $E_i \hookrightarrow F_{i}$ determines a generically surjective chain $F_{i-1}:=E_i  \ra F_i  \ra \cdots \ra F_r$ with invariants $\underline{n}_{\geq i-1}$ and $\underline{d}_{\geq i-1}$. We note that this iteration of Hecke correspondence stacks ends with $\mathcal{C}h_{\geq r}^{\mathrm{gen-surj}} = \mathcal{B}un_{n_r,d_r}$, which is a smooth exhaustive stack by \cite[Theorem 3.2]{HPL}. By repeatedly applying Theorem \ref{thm_hecke} and combining with the formula for the rational motive of $\mathcal{B}un_{n_r,d_r}$ in \cite[Theorem 1.1]{HPL_formula}, we obtain the claimed formulas for $M(\Chgs)$. Finally, we have $M(\Jac(C))\in \langle M(C) \rangle^{\otimes}$ and $\QQ{i}\in \langle  M(C) \rangle^{\otimes}$ for all $i\in\NN$ by \cite[Theorem 4.2.3, Proposition 4.2.5]{AHEW}, and we deduce that $M(\Chgs)$ lies in $\llangle M(C)\rrangle^{\otimes}$.
\end{proof}

\section{The recursive description of the motive of the Higgs moduli space}

In section, we study the motive of the moduli space $\mH$ of semistable Higgs bundles over $C$ of rank $n$ and degree $d$ and prove Theorem \ref{main_thm_intro}.

The scaling $\GG_m$-action on $\mH$ described in $\S$\ref{sec Gm action} has fixed locus equal to a finite disjoint union of moduli spaces of $\alpha_H$-semistable chains by Corollary \ref{cor fixed locus modular}
\[(\mH)^{\GG_m}=\bigsqcup_{(\underline{n}',\underline{d}') \in \cI} \ch^{\alpha_H,ss}_{\underline{n}',\underline{d}'}\]
where the Higgs stability parameter $\alpha_H$ is non-critical for all the invariants $\underline{n}'$ and $\underline{d}'$ appearing in this finite index set $\cI$. Moreover, by Proposition \ref{prop Higgs BB decomp} there is an associated Bia{\l}ynicki-Birula decomposition of $\mH$ and by Theorem \ref{thm_bb_mot}, we obtain the following motivic decomposition
\begin{equation}\label{motivic BB higgs}
M(\mH) \simeq \bigoplus_{(\underline{n}',\underline{d}') \in \cI} M(\ch^{\alpha_H,ss}_{\underline{n}',\underline{d}'})\{ n^2(g-1) +1\}.
\end{equation}

Consequently it suffices to describe the motives of the moduli spaces of $\alpha_H$-semistable chains appearing in this decomposition. Let us outline the strategy for doing this which follows the geometric ideas in \cite{GPHS,GPH, Heinloth_LaumonBDay}. We will also check along the way that the relevant algebraic stacks are smooth and exhaustive.
\begin{enumerate}[(i)]
\item Relate the motives of the moduli spaces $\ch^{\alpha_H,ss}_{\underline{n}',\underline{d}'}=\ch^{\alpha_H,s}_{\underline{n}',\underline{d}'}$ in the above decomposition to the motives of the stacks $\mathcal{C}h^{\alpha_H,s}_{\underline{n}',\underline{d}'}$ by using the fact that $\mathcal{C}h^{\alpha_H,s}_{\underline{n}',\underline{d}'} \ra \ch^{\alpha_H,s}_{\underline{n}',\underline{d}'}$ is a trivial $\GG_m$-gerbe (as discussed in $\S$\ref{sec moduli chains}); see Lemma \ref{lemma mot trivial gerbe}.
\item Use a wall-crossing argument together with a Harder--Narasimhan (HN) recursion to relate for all $\alpha \in \Delta^\circ_r$ the motives of moduli stacks of chains with $\alpha$-HN type $\tau$ (including the trivial HN type $\tau = ss$) to stacks whose motives we can compute (namely stacks of generically surjective chains and the empty stack); see $\S$\ref{sec motivic WC and HN rec}.
\item Use the fact that for all $(\underline{n}',\underline{d}')$ in the decomposition \eqref{motivic BB higgs} the stability parameter $\alpha_H$ is non-critical to perform a slight perturbation to $\widetilde{\alpha}_H \in \Delta^\circ$  for which $\mathcal{C}h^{\alpha_H,ss}_{\underline{n}',\underline{d}'} = \mathcal{C}h^{\widetilde{\alpha}_H ,ss}_{\underline{n}',\underline{d}'}$ (and thus the motive of the stacks appearing in the RHS of \eqref{motivic BB higgs} are described by the second step).
\end{enumerate}

\subsection{Motivic consequences of wall-crossing and HN recursions}\label{sec motivic WC and HN rec}

Let us start with a general proposition which will be required for the proof of the main theorem in this subsection.

\begin{prop}\label{prop inf union} 
Let $\fX$ be a smooth exhaustive stack which admits a countable stratification $\fX = \cup_{i \in \NN} \fX_i$ by locally closed smooth quasi-compact substacks $\fX_i$ such that the closure of $\fX_i$ is contained in the union of higher $\fX_j$ with $j \geq i$. In particular, $\fX_{0}$ is open in $\fX$. Then the motive $M(\fX_0)$ in $\DM^{\eff}(k,R)$ lies in the localising subcategory generated by the motive of $\fX$ and Tate twists of the motives of $\fX_i$ for $i > 0$; that is,
\[ M(\fX_0) \in \llangle M(\fX), M(\fX_i)\{r\} : i > 0 , r \geq  0 \rrangle \]
\end{prop}
\begin{proof}
Write $\cD:=\llangle M(\fX), M(\fX_i)\{r\} : i > 0 , r \geq  0 \rrangle$. For each $i >0$, the open immersion $\fX_0 \hookrightarrow \fX_{\leq i}$ has closed complement $\fX_{(0,i]} := \cup_{0 < j \leq i} \fX_j$ and induces a Gysin distinguished triangle (see \cite[Proposition 2.27 (iii)]{HPL} for a version for smooth exhaustive stacks)
\begin{equation}\label{frank}
M(\fX_0) \ra M(\fX_{\leq i}) \ra C_i:=M(\fX_{(0,i]} )\{\codim(\fX_{(0,i]})\}\rap
\end{equation} 
in $\DM^{\eff}(k,R)$. Since $M(\fX_j) \in \cD$ for all $j>0$ and $C_i$ can be expressed as a successive extension of Tate twists of the motives $M(\fX_j)$ for all $j>0$ by inductively using Gysin triangles, we conclude that $C_i \in \cD$ for all $i > 0$.

We then take the homotopy colimit of the above triangles to obtain a distinguished triangle 
\[ M(\fX_0) \ra \hocolim_i M(\fX_{\leq i}) \ra \hocolim_i C_i\rap.\]
The functoriality of the triangle which is used implicitly here is discussed in \cite[Lemma 2.26]{HPL}
Since $\cD$ is closed under infinite direct sums and taking cones, we deduce that $ \hocolim_i C_i \in \cD$. Moreover, we have $\hocolim_i M(\fX_{\leq i}) \simeq M(\fX)$ by \cite[Lemma 2.26]{HPL}, and $M(\fX)\in \cD$ by definition. This concludes the proof.
\end{proof}

We can relate the motives of stacks of chains of non-trivial HN types to stacks of semistable chains as follows.

\begin{lemma}\label{lemma HN strata motive}
Let $r \in \NN$, $\underline{n}\in \NN^{r+1}$, $\underline{d} \in \ZZ^{r+1}$, $\alpha \in \Delta_r^\circ \subset \RR^{r+1}$ and $\tau = (\underline{n}^j,\underline{d}^j)_{j=1,\dots, l}$ be an $\alpha$-HN type. Then the stack $\Chtau$ is smooth and exhaustive and the motive in $\DM^{\eff}(k,R)$ of the stack $\Chtau$ of length $r$ chains of $\alpha$-HN type $\tau$ with invariants $(\underline{n}, \underline{d})$ is given by
\[ M(\Chtau) \simeq  \bigotimes_{j=1}^l M(\mathcal{C}h_{\underline{n}^j,\underline{d}^j}^{\alpha,ss}) \{ r_\tau\}, \]
where $r_\tau$ is defined in Proposition \ref{prop HN strata geom descr}.
\end{lemma}
\begin{proof}
For $\alpha \in \Delta_{r}^{\circ}$ and every choice of invariants $(\underline{m}, \underline{e})$, the stack $\mathcal{C}h_{\underline{m},\underline{e}}^{\alpha,ss}$ is smooth by \cite[Theorem 3.8 (vi)]{ACGPS}, and a quotient stack via the GIT construction of the corresponding moduli space. It is thus a smooth exhaustive stack by \cite[Lemma 2.16]{HPL}. The result then follows by combining Proposition \ref{prop HN strata geom descr} and the fact that products of exhaustive stacks are exhaustive together with the $\AA^1$-homotopy invariance and the K\"{u}nneth isomorphism in $\DM^{\eff}(k,R)$ (see \cite[Proposition 2.27]{HPL}).
\end{proof}

Finally we implement a wall-crossing argument together with a Harder--Narasimhan recursion following \cite{GPHS,GPH}.

\begin{thm}\label{thm mot of ss chains via WC and HN}
For all $r \in \NN$, $\underline{n}\in \NN^{r+1}$, $\underline{d} \in \ZZ^{r+1}$, $\alpha \in \Delta_r^\circ \subset \RR^{r+1}$ and $\alpha$-HN types $\tau$, the stack $\Chtau$ of length $r$ chains of $\alpha$-HN type $\tau$ with invariants $(\underline{n}, \underline{d})$ fits into an (infinite) collection of distinguished triangles in $\DM^{\eff}(k,R)$ whose other terms are expressed in terms of appropriate Tate twists of
\begin{enumerate}
\item motives of stacks of generically surjective chains, or
\item motives of $\mathcal{C}h_{\underline{n}',\underline{d}'}^{\alpha',ss}$ for $|\underline{n'}| < |\underline{n}|$ and $\alpha' \in \Delta_r^\circ$.
\end{enumerate}
Furthermore, if $R=\QQ$, the motive of the stack $\Chtau$  lies in the localising tensor subcategory of $\DM^{\eff}(k,\QQ)$ generated by the motive of the curve $C$ for all such $r,\underline{n},\underline{d},\alpha$ and $\tau$; that is,
\[ M(\Chtau) \in \llangle  M(C) \rrangle^{\otimes} \subset \DM^{\eff}(k,\QQ). \]
\end{thm}
\begin{proof}
For non-trivial HN types $\tau$, we can use Lemma \ref{lemma HN strata motive} to express their motives in terms of motives of stacks of semistable chains. Hence we suppose that $\tau = ss$ is the trivial HN type corresponding to semistability. We then employ the wall-crossing results of Proposition \ref{prop WC}: there exists a path $(\alpha_t)_{t \geq 0}$ in $\Delta_r^\circ$ with $\alpha_0 = \alpha$ such that
\begin{enumerate}
\item the (finitely many) critical stability parameters $\alpha_{t_1}, \dots \alpha_{t_m}$ on this path lie on single walls and in a neighbourhood of these critical values the path is linear.
\item for $t >\!>0$ the notion of $\alpha_t$-(semi)stability parameter is independent of $t$ and 
\begin{enumerate}[label={\upshape(\alph*)}]
\item if $\underline{n}$ is non-constant, then $\Ch^{\alpha_t,ss}=\emptyset$, or
\item  if $\underline{n}$ is constant, then $\Ch^{\alpha_t,ss} \subset \Chgs$ and $\Chgs$ is an infinite union of $\alpha_t$-HN strata.
\end{enumerate}
\end{enumerate}
Let us write $\alpha_\infty$ for $\alpha_t$ with $t >\!>0$ as required above. Then this path involves finitely many different notions of (semi)stability for the parameters
\[ \alpha_0, \alpha_{t_1 - \epsilon}, \alpha_{t_1}, \alpha_{t_1 + \epsilon}, \dots, \alpha_{t_m - \epsilon}, \alpha_{t_m}, \alpha_{t_m + \epsilon}=\alpha_{\infty}. \]
At each wall-crossing $\alpha_{t_i - \epsilon}, \alpha_{t_i}, \alpha_{t_i + \epsilon}$, we have by \eqref{wall-crossing} wall-crossing decompositions
\[ \Ch^{\alpha_{t_i},ss} = \Ch^{\alpha_{t_i\pm \epsilon},ss} \sqcup \bigsqcup_{\tau \in I_{\pm}} \Ch^{\alpha_{t_i\pm \epsilon},\tau}, \]
with finite index sets $I_\pm$ which are partially ordered so that the closure of a given stratum is contained in the union of all higher strata. In particular, a maximal index $\tau_{\pm}^{\max}$ corresponds to a HN stratum $\Ch^{\alpha_{t_i\pm \epsilon},\tau_{\pm}^{\max}}$ which is closed in $\Ch^{\alpha_{t_i},ss}$; hence, by \cite[Proposition 2.27 (iii)]{HPL}, there is a Gysin distinguished triangle associated to the closed immersion $\Ch^{\alpha_{t_i\pm \epsilon},\tau_{\pm}^{\max}} \hookrightarrow \Ch^{\alpha_{t_i},ss}$ of smooth exhaustive stacks. By iterating this procedure, we obtain a diagram of Gysin distinguished triangles relating $M(\Ch^{\alpha_0,ss})$ and $M(\Ch^{\alpha_{\infty },ss})$ whose other terms are motives of stacks of chains of non-trivial HN types (which can be described in terms of motives of stacks of semistable chains for smaller invariants). Therefore, it suffices to describe the motive of $\Ch^{\alpha_\infty,ss}$. This is split into the two cases described above.
\begin{enumerate}[label={\upshape(\alph*)}]
\item If $\underline{n}$ is non-constant, then $\Ch^{\alpha_\infty,ss}=\emptyset$ (in which case the motive of this stack is zero).
\item If $\underline{n}$ is constant, then $\Ch^{\alpha_\infty,ss} \subset \Chgs$ and moreover, $\Chgs$ is an infinite union of $\alpha_\infty$-HN strata. We then obtain another (infinite) diagram of motives which relates $\Ch^{\alpha_\infty,ss}$ and $\Chgs$ and whose other terms involve Tate twists of products of motives of stacks $\mathcal{C}h_{\underline{n}',\underline{d}'}^{\alpha_\infty,ss}$ for smaller invariants $|\underline{n'}| < |\underline{n}|$.
\end{enumerate}
For the final claim, we apply Proposition \ref{prop inf union} to $ \Chgs = \sqcup_{\tau \in \cJ} \Ch^{\alpha_\infty,\tau}$, where the lowest open stratum in this decomposition is given by $\tau = ss$. Note that strictly speaking Proposition \ref{prop inf union} concerned $\NN$-indexed stratifications, but this can equally be applied to stratifications indexed by a partially ordered set with a cofinal copy of $\NN$; for instance, one can filter $\cI$ by maximal slope to get such a cofinal $\NN$. By Proposition \ref{prop gen surj}, the motive of $\Chgs$ for constant $\underline{n}$ lies in the category $\llangle  M(C) \rrangle^{\otimes}$ . By induction, we have that for all non-trivial HN types $\tau$, the motive of $\Ch^{\alpha_\infty,\tau}$ lies in $\llangle  M(C) \rrangle^{\otimes}$. Hence, we can apply Proposition \ref{prop inf union} to conclude that $M(\Ch^{\alpha_\infty,ss})\in\llangle  M(C) \rrangle^{\otimes}$, which by the previous wall-crossing arguments is enough to carry out an induction on the invariants of the chains and complete the proof.
\end{proof}

\subsection{The motive of the Higgs moduli space is built from the motive of the curve}\label{sec final proof}

We can now prove the first part of Theorem \ref{main_thm_intro}.

\begin{thm}\label{thm main first}
Assume that $C(k) \neq \emptyset$ and that $R$ is a $\QQ$-algebra. Then the motive $M(\mH)$ lies in the thick tensor subcategory $\langle M(C)\rangle^{\otimes}$ of $\DM^{\eff}_{c}(k,R)$ generated by $M(C)$.  
\end{thm}

\begin{proof}

The result for $R=\QQ$ implies the same for any $\QQ$-algebra by extension of scalars. We can thus assume $R=\QQ$. This is only necessary to invoke results from \cite{HPL_formula} which were formulated with $R=\QQ$, but the proof in \textit{loc.\ cit.\ } applies in fact just as well for $R$ a $\QQ$-algebra.
  
By Lemma \ref{lemma:compact} below applied with $M=M(C)$, it is enough to show that $M(\mH)$ lies in the subcategory $\cC:=\llangle M(C) \rrangle^{\otimes}$.
  
By the motivic Bia{\l}ynicki-Birula decomposition \eqref{motivic BB higgs} of $\mH$, it suffices to show that the motives of the moduli spaces of $\alpha_H$-semistable chains $\ch^{\alpha_H,ss}_{\underline{n}',\underline{d}'}$ in this decomposition lie in $\cC$. By Corollary \ref{cor fixed locus modular}, the Higgs stability parameter $\alpha_H$ is non-critical for all invariants $(\underline{n}',\underline{d}')$ appearing in this decomposition and thus $\alpha_H$-semistability coincides with $\alpha_H$-stability.

Since we assumed $C(k) \neq \emptyset$, it follows that the stack $\mathcal{C}h^{\alpha_H,s}_{\underline{n}',\underline{d}'}$ is a trivial $\GG_m$-gerbe over its coarse moduli space $\ch^{\alpha_H,ss}_{\underline{n}',\underline{d}'}$ (see $\S$\ref{sec moduli chains}). Hence, by Lemma \ref{lemma mot trivial gerbe} below, it suffices to show that the motives of the stacks $\mathcal{C}h^{\alpha_H,s}_{\underline{n}',\underline{d}'}$ lie in $\cC$.

We cannot directly apply Theorem \ref{thm mot of ss chains via WC and HN} to describe the motive of this stack, as the Higgs stability parameter $\alpha_H$ lies on the boundary of the cone $\Delta_r$. However, as $\alpha_H$ is non-critical for the invariant appearing in the Bia{\l}ynicki-Birula decomposition, this stability parameter does not lie on a wall and so we can choose a slight perturbation $\widetilde{\alpha}_H$ of $\alpha_H$ which lies in $\Delta_r^0$ and is in the same chamber as $\alpha_H$ (so that $\widetilde{\alpha}_H$ determines the same notion of stability as $\alpha_H$). Then we have $\mathcal{C}h^{\alpha_H,s}_{\underline{n}',\underline{d}'} = \mathcal{C}h^{\widetilde{\alpha}_H,s}_{\underline{n}',\underline{d}'}$ and the motive of the latter lies in $\cC$ by Theorem \ref{thm mot of ss chains via WC and HN}.
\end{proof}

\begin{lemma}\label{lemma:compact}
Let $M\in \DM^{\eff}_{c}(k,R)$ be a compact effective motive. Then we have
\[
\langle M \rangle^{\otimes} = \llangle M \rrangle^{\otimes}\cap \DM^{\eff}_{c}(k,R)
\]
as full subcategories of $\DM^{\eff}(k,R)$.
\end{lemma}
\begin{proof}
First, let us show that
  \[
\llangle M \rrangle^{\otimes} = \llangle  M^{\otimes n},\ n\geq 0 \rrangle
\]
We have $\llangle M^{\otimes n},\ n\geq 0\rrangle \subset \llangle M \rrangle^{\otimes}$ since the right-hand side contains $M^{\otimes n}$ for all $n\geq 0$ and is triangulated and stable by small direct sums. Recall from the section on Notations and Conventions that the category $\llangle M^{\otimes n},n\geq 0\rrangle$ can be written as a transfinite union of full subcategories $\llangle M^{\otimes n}, n\geq 0 \rrangle^{\alpha}$ with $\alpha$ an ordinal, where
\begin{itemize}
\item $\llangle M^{\otimes n}, n\geq 0 \rrangle^{0}$ is the full subcategory on the set $\{M^{\otimes n}[k]|n\in \NN, k\in \ZZ\}$, and
\item for all $\alpha>0$, $\llangle M^{\otimes n}, n\geq 0 \rrangle^{\alpha}$ is the full subcategory of objects which are extensions or small direct sums of objects in $\cup_{\beta<\alpha} \llangle M^{\otimes n}, n\geq 0\geq  \rrangle^{\beta}$.
\end{itemize}
Using this description, the fact that the set $\{M^{\otimes n}[k]|n\in \NN, k\in \ZZ\}\cap\{0\}$ is stable under tensor product and the fact that tensor products commute with small direct sums in $\DM^{\eff}(k,R)$, an transfinite induction implies that $\llangle M^{\otimes n},n\geq 0\rrangle$ is stable by tensor products. This shows the converse inclusion $\llangle M \rrangle^{\otimes}\subset\llangle M^{\otimes n},\ n\geq 0\rrangle$. A variant of the above argument, replacing small sums by direct factors, establishes the equality
\[
\langle M \rangle^{\otimes} = \langle  M^{\otimes n},\ n\geq 0 \rangle.
\]
So the statement of the lemma is equivalent to 
\[
  \langle M^{\otimes n},n\geq 0 \rangle = \llangle M^{\otimes n},n\geq 0\rrangle\cap \DM^{\eff}_{c}(k,R)
\]
Since $M$ is assumed compact and compact objects are stable by tensor products in $\DM^{\eff}_{c}(k,R)$, $\{M^{\otimes n}| n\geq 0\}$ is a set of compact objects in a compactly generated triangulated category. The subcategory $\llangle M^{\otimes n},n\geq 0\rrangle^{\otimes}$ is thus also compactly generated, and an object in $\DM^{\eff}_{c}(k,R)\cap \llangle M^{\otimes n},n\geq 0\rrangle^{\otimes}$ is also compact in $\llangle M^{\otimes n},n\geq 0\rrangle^{\otimes}$. By \cite[Proposition 2.1.24]{Ayoub_these_1} applied to the compactly generated subcategory $\llangle M^{\otimes n},n\geq 0\rrangle^{\otimes}$, we deduce that
\[
  \llangle M^{\otimes n},n\geq 0\rrangle\cap \DM^{\eff}_{c}(k,R)\subset \llangle M^{\otimes n},n\geq 0\rrangle_{c} = \langle M^{\otimes n},n\geq 0 \rangle.
\]
Since the other inclusion is immediate, this concludes the proof.
\end{proof} 

\begin{lemma}\label{lemma mot trivial gerbe}
Let $\fX \ra \cY$ be a morphism of stacks which is a trivial $\GG_m$-gerbe; then
\[ M(\fX) \simeq M(\cY) \otimes M(B\GG_m) \simeq M(\cY) \otimes \bigoplus_{j \geq 0 } \QQ\{ j \}. \]
In particular, if $\cD$ is a localizing subcategory of $\DM^{\eff}(k,R)$ stable by Tate twists, then $M(\fX)$ lies in $\cD$ if and only if $M(\cY)$ lies in $\cD$.
\end{lemma}
\begin{proof}
This follows from the K\"{u}nneth isomorphism \cite[Proposition 2.27 (i)]{HPL} and \cite[Example 2.21]{HPL}.
\end{proof}

In fact by Theorem \ref{thm mot of ss chains via WC and HN}, the motive of the stacks $\mathcal{C}h^{\alpha_H,s}_{\underline{n}',\underline{d}'}$ lying over the moduli spaces $\ch^{\alpha_H,s}_{\underline{n}',\underline{d}'}$ appearing in the motivic Bia{\l}ynicki-Birula decomposition of $M(\mH)$ all can be described by an (infinite) collection of distinguished triangles in $\DM^{\eff}(k,R)$ whose other terms are Tate twists of tensor products of motives of stacks of semistable chains with smaller invariants (which can be inductively described) or motives of stacks of generically surjective chains (which are described by Theorem \ref{thm_hecke}). In particular, the motive of $\mathcal{C}h^{\alpha_H,s}_{\underline{n}',\underline{d}'}$ fits into an infinite collection of distinguished triangles in $\DM^{\eff}(k,R)$ whose terms are built out of tensor products and direct sums of motives of $C$, its symmetric powers, its Jacobian and Tate twists. Unfortunately, we cannot deduce a formula for the motives of the moduli spaces $\ch^{\alpha_H,ss}_{\underline{n}',\underline{d}'}$ by ``canceling'' the factor of the motive of $B\GG_m$. However, instead if we let $\cH^{ss}_{n,d}$ denote the stack of semistable Higgs bundles, we can describe the motive of $\cH^{ss}_{n,d}$ using these triangles. Under our assumptions that $n$ and $d$ are coprime and $C(k) \neq \emptyset$, the stack $\cH^{ss}_{n,d}$ is a trivial $\GG_m$-gerbe over its coarse moduli space $\mH$ (see \cite[Lemma 3.10]{heinloth_lectures}) and we obtain the following result.

\begin{cor}\label{cor motive Higgs stack}
Assume that $C(k)\neq \emptyset$ and that $R$ is a $\QQ$-algebra. The motive of the stack $\cH^{ss}_{n,d}$ in $\DM^{\eff}(k,R)$ fits into an explicit finite sequence of distinguished triangles whose other terms are built out of tensor products and direct sums of motives of $C$, its symmetric powers, its Jacobian and Tate twists.
\end{cor}

\subsection{Corollaries of purity}\label{subsec purity}

We refer to \cite[Definition 1.1]{Wildeshaus_proj} for the definition of weight structures on triangulated categories in the sense of Bondarko. There are two natural conventions for weight structures, both of which occur in the literature, and we use the homological one used in \textit{loc.\ cit.\ }. We will also use the notion of a bounded weight structure \cite[Definition 1.2.1.6]{Bondarko}. Recall that the triangulated category $\DM^{\eff}_{c}(k,R)$ of constructible effective motives carries a bounded weight structure $(\DM^{\eff}_{c}(k,R)_{w\geq 0},\DM^{\eff}_{c}(k,R)_{w\leq 0})$ whose heart $\DM^{\eff}_{c}(k,R)_{w\geq 0}\cap\DM^{\eff}_{c}(k,R)_{w\leq 0}$ is equivalent to the category of effective Chow motives $\Chow^{\eff}(k,R)$ over $k$ via Voevodsky's embedding \cite[\S 6.5-6]{Bondarko}. We call this weight structure the Chow weight structure on $\DM^{\eff}_{c}(k,R)$, and we systematically identify effective Chow motives with objects in $\DM^{\eff}_{c}(k,R)$. Objects in the heart of the weight structures are called pure motives.

Let $\cT$ be a triangulated category and $\cT'\subset \cT$ be a triangulated subcategory. Suppose that $\cT$ is equipped with a weight structure $(\cT_{w\geq 0}, \cT_{w\leq 0})$. We say that the weight structure restricts to $\cT'$ if $(\cT_{w\geq 0}\cap \cT',\cT_{w\leq 0}\cap \cT')$ is a weight structure on $\cT'$.

\begin{lemma}\label{lemma:weight-structure}
Let $X$ be a smooth projective variety over $k$. The Chow weight structure on $\DM^{\eff}_{c}(k,R)$ restricts to the tensor triangulated subcategory $\langle  M(X) \rangle^{\otimes}$. The heart of the restricted weight structure is the idempotent complete additive tensor subcategory $\Chow^{\eff}_{X}(k,R)$ generated by $M(X)$. \end{lemma}

\begin{proof}
This is an variant of a result of Wildeshaus \cite[Proposition 1.2]{wildeshaus_Picard} with essentially the same proof. For the convenience of the reader, we reproduce the argument.

The category $\langle  M(X) \rangle^{\otimes}$ is generated by $\Chow^{\eff}_{X}(k,R)$ as a triangulated category. Moreover, $\Chow^{\eff}_{X}(k,R)$ is a subcategory of $\Chow^{\eff}(k,R)$ which is the heart of a weight structure, hence is negative in the sense of \cite[Definition 4.3.1.1]{Bondarko}. Furthermore, $\Chow^{\eff}_{X}(k,R)$ is idempotent complete by construction. By \cite[Theorem 4.3.2, II.1-2]{Bondarko}, there exists a bounded weight structure $(\langle  M(X) \rangle^{\otimes}_{w\geq 0},\langle  M(X) \rangle^{\otimes}_{w\leq 0})$ whose heart is precisely $\Chow^{\eff}_{X}(k,R)$.

By construction, we have $\langle  M(X) \rangle^{\otimes}_{w=0}\subset \DM^{\eff}_{c}(k,R)_{w=0}$, and since the corresponding weight structures are both bounded, an inductive argument using weight decompositions implies that we have both
\[\langle  M(X) \rangle^{\otimes}_{w\geq 0}\subset \DM^{\eff}_{c}(k,R)_{w\geq 0}\cap \langle M(X) \rangle^{\otimes}\]
and
\[\langle  M(X) \rangle^{\otimes}_{w\leq 0}\subset \DM^{\eff}_{c}(k,R)_{w\leq 0}\cap \langle M(X) \rangle^{\otimes}.\]
It remains to prove that these inclusions are equalities. Let $M\in \DM^{\eff}_{c}(k,R)_{w\geq 0}\cap \langle M(X) \rangle^{\otimes}$. Consider a weight decomposition of $M$ with respect to the weight structure on $\langle  M(X) \rangle^{\otimes}$ constructed above, say
\[
N \stackrel{f}{\to} M \to P\rap
\]
with $N\in \langle  M(X) \rangle^{\otimes}_{w\leq -1}$ and $P\in \langle  M(X) \rangle^{\otimes}_{w\geq 0}$. In particular, $N\in \DM^{\eff}_{c}(k,R)_{w\leq -1}$. By the orthogonality property for the Chow weight structure, the morphism $f$ is $0$, and thus $M$ is a retract of $P$, so it is also in $\langle  M(X) \rangle^{\otimes}_{w\geq 0}$. The argument for negative weights is similar.
\end{proof}  

We can now complete the proof of Theorem \ref{main_thm_intro}.

\begin{cor}\label{cor:direct-factor}
Assume that $C(k) \neq \emptyset$ and that $R$ is a $\QQ$-algebra. Then $M(\mH)$ can be written as a direct factor of the motive of a large enough power of $C$. In particular, $M(\mH)$ is a pure abelian motive.
\end{cor}
\begin{proof}
By Theorem \ref{thm main first}, Proposition \ref{prop Higgs BB decomp} and Corollary \ref{cor semiproj pure}, the motive $M(\mH)$ belongs to $\langle M(C) \rangle^{\otimes}\cap \Chow^{\eff}(k,R)$. This latter category is the heart $\Chow^{\eff}_{X}(k,R)$ of the restricted weight structure of Lemma \ref{lemma:weight-structure}, and objects of that heart are direct factors of the motives of powers of $C$ by construction.
\end{proof}

\appendix

\section{Motivic Bia{\l}ynicki-Birula decompositions}
\label{sec mot BB}
\subsection{Geometric Bia{\l}ynicki-Birula decompositions}

Let $X$ be a smooth projective $k$-variety equipped with a $\GG_m$-action. By a result of Bia{\l}ynicki-Birula \cite{BB} and Hesselink \cite{Hesselink}, there exists a decomposition of $X$, indexed by the connected components of the fixed locus $X^{\GG_m}$, with very good geometric properties. In fact, this decomposition exists in the following slightly more general context. 

\begin{defn}\label{def semiproj Gm}
A $\GG_m$-action on a smooth quasi-projective $k$-variety $X$ is \emph{semi-projective} if 
\begin{itemize}
\item $X^{\GG_m}$ is proper (and thus projective), and
\item for every point $x\in X$ (not necessarily closed), the action map $f_x:\GG_m\ra X$ given by $t\mapsto t \cdot x$ extends to a map $\bar{f}_x:\AA^1\ra X$. Since $X$ is separated, the extension is unique and we write $\lim_{t \ra 0} t \cdot x$ for the limit point $\bar{f}_x(0)\in X$.
\end{itemize}
\end{defn}

In particular, any $\GG_m$-action on a smooth projective variety is semi-projective. Note that the limit point $\lim_{ t\ra 0} t\cdot x$ is necessarily a fixed point of the $\GG_m$-action if it exists.

\begin{thm}[Bia{\l}ynicki-Birula]\label{thm_bb_dec}
Let $X$ be a smooth quasi-projective variety over $k$ with a semi-projective $\GG_m$-action. Then the following statements hold.
\begin{enumerate}[label={\upshape(\roman*)}]
\item\label{fixed} The fixed locus $X^{\GG_m}$ is smooth and projective. Write $\{X_i\}_{i\in I}$ for its set of connected components and $d_i$ for the dimension of $X_i$.
\item\label{strata} For $i\in I$, write $X^+_i$ for the attracting set of $X_i$, i.e., the set of all points $x\in X$ such that $\lim_{t\ra 0} t\cdot x\in X_i$. Then $X^+_i$ is a locally closed subset of $X$ and $X=\coprod_{i\in I} X^+_i$. 
\item\label{bundle} For every $i\in I$, the map of sets $X^+_i\ra X_i$ given by $x\mapsto \lim_{t\ra 0}t\cdot x$ underlies a morphism of schemes $p_i^+:X_i^+\ra X_i$, which is a Zariski locally trivial fibration in affine spaces. For each $i \in I$, we have
  \[
\dim(X)=d_i+c_i^++r_i^+
    \]
where $c_i^+ = \codim_X(X_i^+)$ and $r_i^+$ denotes the rank of $p_i^+$.
\item\label{tangent} The tangent space $T_xX$ of a fixed point $x\in X_i$ admits a $\GG_m$-action, hence a weight space decomposition $T_xX=\bigoplus_{k\in\ZZ}(T_x X)_k$. Then we have $T_xX_i=(T_xX)_0$ and $N_{X_i/X^+_i}\simeq \bigoplus_{k>0}(T_xX)_k$ and $(N_{X^+_i/X})_{|X_i}\simeq \bigoplus_{k<0}(T_xX)_k$.
\item\label{order} Let $n := |I|$; then there is a bijection $\varphi : \{ 1, \dots, n \} \ra I$ and a filtration of $X$ by closed subschemes
  \[
\emptyset=Z_n \subset Z_{n-1}\subset \ldots\subset Z_0=X
\]
such that, for all $1\leq k\leq n$, we have that $Z_{k-1} - Z_{k}=X_{\varphi(k)}^+$ is a single attracting set (and thus, in particular, is smooth).
\end{enumerate}  
\end{thm}
\begin{proof}
Points \ref{fixed} - \ref{tangent} are all established in \cite[Theorem 4.1]{BB} under the assumption that $k$ is algebraically closed (the hypothesis that $X$ is smooth and quasi-projective is used to ensure the existence of an open covering by $\GG_m$-invariant affine subsets, and the assumption that $X$ is semi-projective implies that the strata in \cite[Theorem 4.1]{BB} cover all of $X$). The hypothesis that $k$ is algebraically closed is removed by Hesselink in \cite{Hesselink}. 

As the proof of \ref{order} is scattered through \cite[\S 1]{Hausel_RV_semiproj}, we recapitulate their argument. Let $L$ be a very ample line bundle on the quasi-projective variety $X$. By \cite[Theorem 1.6]{Sumihiro} applied to the smooth (hence normal) variety $X$, there exists an integer $n\geq 1$ such that $L^{\otimes n}$ admits a $\GG_m$-linearisation. In particular, this provides a projective space $\PP$ with a linear $\GG_m$-action and a $\GG_m$-equivariant immersion $\iota:X\ra \PP$. Let $\{\PP_j\}_{j \in J}$ be the connected components of $\PP^{\GG_m}$ with corresponding attracting sets $\PP_j^+$ for each $j \in J$; then by equivariance of $\iota$, there is a (not necessarily injective) map $\tau: I \ra J$ such that $\iota(X_i^+) \subset \PP_{\tau(i)}^+$ for all $i \in I$.

As each $X_i$ is connected, the group $\GG_m$ acts on $L_{|X_i}$ via a character $\omega_i\in \Hom(\GG_m,\GG_m)\simeq \ZZ$. For the partial order on $I$ given by $i < i'\Leftrightarrow \omega_i > \omega_{i'}$, we claim that for $i\neq i'\in I$
\begin{equation}\label{filtr}
\overline{X_i^+} \cap X_{i'}^+ \neq \emptyset \: \text{ only if } i' > i.
\end{equation}
Indeed, we can similarly define a partial order on $J$ such that $i < i'$ if and only if $\tau(i) < \tau(i')$ by equivariance of $\iota$; then one can easily deduce that \eqref{filtr} holds for $\PP$ from the linearity of the $\GG_m$-action on $\PP$. 
We now deduce \eqref{filtr} for $X$ from the corresponding ambient property for $\PP$; the only non-trivial case to consider is when $i \neq i'$ have the same image $j$ under $\tau$, so that $X_i^+$ and $X_{i'}^+$ are both contained in $\PP_j^+$. In this case, if $x \in \overline{X_i^+} \cap X_{i'}^+$, then by passing to an algebraic closure of $k$ if necessary, we can assume that there is a connected curve $S \subset X$ with $x \in S$ and $S - x \subset X_i^+$; then $S \subset \PP_j^+$ and as the action on $X$ is semi-projective, $p_j^+(S) \subset X^{\GG_m}$ and this connects $X_i$ and $X_{i'}$, contradicting $i \neq i'$.

Finally to prove the filterability of $X$, we choose any total ordering of $I$ extending the above partial order and we view this ordering as a bijection $\varphi : \{ 1, \dots, n \} \ra I$. Then for $0 \leq k \leq n$, 
\[ Z_k:= \bigcup_{\begin{smallmatrix} i \in I :\\ \varphi^{-1}(i) > k\end{smallmatrix}} X_i^+\]
is closed in $X$ by \eqref{filtr} with $Z_n = \emptyset$ and $Z_0 = X$.
\end{proof}

\begin{rmk}
Let $X$ be smooth projective with a fixed $\GG_m$-action $(t,x)\mapsto t \cdot x$. The opposite $\GG_m$-action $(t,x)\mapsto t^{-1} \cdot x$ has the same fixed point locus as the original action, but the associated Bia{\l}ynicki-Birula decomposition is different. We write $X_i^-$ for the associated strata, $c^-_i$ (resp.\ $r^-_i$) for their codimension (resp.\ their rank as affine bundles), etc. By Theorem \ref{thm_bb_dec} \ref{tangent}, we see that $c^-_i=r^+_i$ and $r^-_i=c^+_i$, and that the strata $X^+_i$ and $X^-_i$ intersect transversally along $X_i$.
\end{rmk}  

\subsection{Motivic consequences}

In this appendix, we let $R$ be a coefficient ring such that the exponential characteristic of $k$ is invertible in $R$.

Let $X$ be a smooth quasi-projective variety with a semi-projective $\GG_m$-action. The geometry exhibited in the previous sections implies a decomposition of the motive of $X$. There are in fact two natural such decompositions, one for the motive $M(X)$ and one for the motive with compact support $M^c(X)$. These motivic decompositions have been studied in \cite{Brosnan, Choudhury_Skowera, Karpenko}; we explain and expand upon their results in this section. Recall that for two smooth $k$-schemes $X$ and $Y$ with $X$ of dimension $d$ and an integer $i\in \NN$, there is an isomorphism
\[ \CH_i(X\times Y)_R \simeq \Hom_{\DM}(M(X),M^c(Y)\{d-i\})\]
with $\CH_i$ the Chow groups of cycles of dimension $i$; when this does not lead to confusion, we use the same notation for a cycle and the corresponding map of motives.

\begin{thm}\label{thm_bb_mot}
Let $X$ be a smooth quasi-projective variety with a semi-projective $\GG_m$-action. With the notation of Theorem \ref{thm_bb_dec}, for each $i \in I$, we let $\gamma^+_i$ be the class of the algebraic cycle given by the closure  $\overline{\Gamma_{p_i^+}}$ of the graph of $p_i^+:X^+_i\ra X_i$ in $X\times X_i$, and we let $(\gamma_i^+)^t$ be the class of the transposition of this graph closure. Then we have the following motivic decompositions (where we use without comment that $M(X_i)\simeq M^c(X_i)$ as $X_i$ is projective).
\begin{enumerate}[label={\upshape(\roman*)}]
\item \label{mot BB} (Bia{\l}ynicki-Birula decomposition for the motive): There is an isomorphism
  \[
M(X)\simeq \bigoplus_{i\in I} M(X_i)\{c^+_i\}.
\]
induced by the morphisms $M(X)\ra M^c(X_i)\simeq M(X_i)\{c^+_i\}$ given by the classes $\gamma^+_i$ for each $i\in I$.
\item \label{comp mot BB} (Bia{\l}ynicki-Birula decomposition for the compactly supported motive): There is an isomorphism 
\[
\bigoplus_{i\in I} M(X_i)\{r^+_i\}\simeq M^c(X).
\]
induced by the morphisms $M(X_i)\{r^+_i\}\ra M^c(X)$ given by the classes $(\gamma_i^+)^t$ for each $i \in I$.
\item \label{Poincare BB} The Poincar\'e duality isomorphism
\[
M^c(X)\simeq M(X)^\vee\{d\}
\]
identifies the motivic Bia{\l}ynicki-Birula decomposition of $M^c(X)$ from \ref{comp mot BB} with the dual of the motivic Bia{\l}ynicki-Birula decomposition of $M(X)$ from \ref{mot BB}. In other words, for every $(i,j)\in I^2$, the composite map
\[
M(X_i)\{r_i^+\}\stackrel{(\gamma^+_i)^t}{\ra} M^c(X)\simeq M(X)^\vee\{d\} \stackrel{(\gamma^+_j)^\vee}{\ra} M(X_j)^\vee\{d-c_j^+\}
\]
is $0$ if $i\neq j$ and is a twist of the Poincar\'e duality isomorphism for the motive of the smooth projective variety $X_i$ if $i=j$ (noting the equality $d-c^+_i-r_i^+=d_i$).  
\end{enumerate}

\end{thm}
\begin{proof}
We first prove \ref{mot BB}. By Theorem \ref{thm_bb_dec}, there is a filtration $\emptyset=Z_n \subset Z_{n-1}\subset \ldots\subset Z_0=X$ by closed subvarieties such that, for all $1\leq j\leq n$, we have that $Z_{j-1} - Z_{j}=X_{j}^+$ is an attracting cell. Let $U_j:= X - Z_j$, which is an open subset and so in particular is smooth. For $1\leq i\leq j\leq n$, let us write $\gamma^+_{i,j}$ for the closure of $\Gamma_{p_i}$ in $U_j\times X_i$ (this makes sense since $X_i^+\subset U_i\subset U_j$) so that $\gamma^+_i = \gamma^+_{i,n}$. We will prove, by induction on $1\leq j\leq n$, that the map
\[
\bigoplus_{i=1}^j\gamma^+_{i,j}: M(U_j)\ra \bigoplus_{1\leq i\leq j} M(X_i)\{c_i^+\}
  \]
is an isomorphism. For $j=1$, the statement holds trivially as $U_0 = \emptyset$ and so $M(U_1) = M(X_1^+)\simeq M(X_1)$ via $p^+_1$. 
Assume that the statement is true for $j-1$. We have a closed immersion $i_{j}:X_j^+\ra U_{j}$ between smooth schemes with codimension $c_j^+$ and open complement $U_{j-1}$; hence, there is a Gysin triangle
\[
  M(U_{j-1})\rightarrow M(U_j)\stackrel{\mathrm{Gy}(i_{j})}{\rightarrow} M(X_j^+)\{c_j^+\}\stackrel{+}{\rightarrow}
\]
for $1 \leq j \leq n$. Since $(\gamma^+_{i,j})|_{U_{j-1}\times X_i} = \gamma^+_{i,j-1}$, the following diagram commutes
  \[
    \xymatrix{ M(U_{j-1}) \ar[r]  \ar[d]^{\wr}_{\oplus_{i=1}^{j-1} \gamma^+_{i,j-1}} & M(U_j) \ar[ld]^{\oplus_{i=1}^{j-1} \gamma^+_{i,j}} \\ 
\bigoplus_{i=1}^{j-1} M(X_i)
      }
    \]
where the left vertical map is an isomorphism by induction. This shows that the triangle splits.

As $p_j^+ : X_j^+ \ra X_j$ is a Zariski locally trivial fibration of affine spaces, $M(p_j^+)$ is an isomorphism. It remains to show that the composition $ M(p_{j})\{c^+_{j}\}\circ \Gy(i_{j}) : M(U_j) \ra M(X_j)\{c^+_j\}$ 
coincides with the map $M(U_{j})\ra M(X_{j})\{c_{j}^{+}\}$ induced by $\gamma^+_{j,j}$. Let us write $\gamma^\circ_{j,j}$ for the graph of $p^{+}_{j}$ considered as a subscheme of $X_{j}^+\times X_{j}$.

Let us recall the functoriality of the Poincar\'e duality isomorphism with respect to algebraic cycles. Let $Y_1,Y_2$ be smooth projective varieties of dimensions $d_1$ and $d_2$, and $\gamma\in \CH^c(Y_1\times Y_2)$, which induces morphisms $\gamma:M(Y_1)\ra M(Y_2)\{c-d_2\}$ and $\gamma^t:M(Y_2)\ra M(Y_1)\{c-d_1\}$. Then the following diagram is commutative
\[
  \xymatrix{
    M(Y_1) \ar[r]^{\gamma \quad \quad} \ar[d]_{\wr} & M(Y_2)\{c-d_2\} \ar[d]_{\wr}\\
    M(Y_1)^\vee\{d_1\} \ar[r]^{(\gamma^t)^\vee\{c\}} & M(Y_2)^\vee \{c\}.
    }
  \]
From this commutativity, it suffices to show that $\gamma^\circ_{j,j} \circ \text{pr}_{1}^{*}\Gy(i_{j}) : M(U_{j}\times X_{j}) \ra R(c^{+}_{j}+d_{j})$ coincides with the map $\gamma^{+}_{j,j}: M(U_{j}\times X_{j}) {\ra} R(c^{+}_{j}+d_{j})$. Let us denote by $a_{j}:\gamma^{+}_{j,j}\ra X^{+} \times X_{j}$ and $b_{j}:\gamma^{+}_{j,j}\ra U_{j}\times X_{j}$ the closed immersions, so that $b_{j} = (i_{j}\times X_{j})\circ a_{j}$. Consider the diagram
    \[
      \xymatrix{
    M(U_{j}\times X_{j}) \ar[dd]_{\Gy(b_{j})} \ar[rr]^{\gamma^{+}_{j,j}} \ar[dr]_{\Gy(i_{j}\times X_{j})} & & R(c^{+}_{j}+d_{j}) \ar@{=}[dd] \\
    & M(X^{+}_{j}\times X_{j})\{c^{+}_{j}\} \ar[ld]^{\Gy(a_{j})\{c^{+}_{j}\}} \ar[ur]_{\gamma^\circ_{j,j}} & \\
    M(\gamma^{+}_{j,j})\{c^{+}_{j}+d_{j}\} \ar[rr]_{M(\gamma\circ_{j,j}\ra \Spec(k))} & &  R(c^{+}_{j}+d_{j})
  }
\]
in $\DM(k,R)$. The left triangle commutes because of the general behaviour of Gysin maps with respect to composition \cite[Theorem 1.34]{Deglise_Gysin_1}. The outer square and the bottom quadrilateral commute because of the compatibility of Gysin maps with fundamental classes of cycles of smooth subvarieties in motivic cohomology \cite[Lemma 3.3]{Deglise_mod_hom}. This implies that the top triangle commutes. Since $\mathrm{pr}_{1}$ is a smooth morphism, we have $\mathrm{pr}_{1}^{*}\Gy(i_{j}) = \Gy(i_{j}\times X_{j})$ by \cite[Proposition 1.19 (1)]{Deglise_Gysin_1} and the commutation of the top triangle is exactly the equality we want. This concludes the proof of \ref{mot BB}.

Statements \ref{comp mot BB} and \ref{Poincare BB} are deduced from \ref{mot BB} by applying Poincar\'e duality and using the functoriality of the Poincar\'e duality isomorphism with respect to algebraic cycles recalled above in the proof of \ref{mot BB}.
\end{proof}  

Since the fixed loci are smooth projective varieties, their motives are pure and we obtain the following corollary.

\begin{cor}\label{cor semiproj pure}
Let $X$ be a quasi-projective variety with a semi-projective $\GG_m$-action; then the motive of $X$ is pure, i.e., it lies in the heart of the weight structure on $\DM^{\eff}(k,R)$ recalled at the beginning of Section \ref{subsec purity}.
\end{cor}

In the smooth projective case, one has $M^c(X) \simeq M(X)$ and one would like compare this decomposition with the decomposition obtained for the opposite $\GG_m$-action. We do not know the answer and leave it as a question.

\begin{question}
 Let $X$ be a smooth projective variety with a $\GG_m$-action. Then, via the isomorphism $M^c(X)\simeq M(X)$, do the motivic Bia{\l}ynicki-Birula decompositions of $M(X)$ in Theorem \ref{thm_bb_mot} \ref{mot BB} and of $M^c(X)$ in Theorem \ref{thm_bb_mot} \ref{comp mot BB} for the \emph{opposite} $\GG_m$-action coincide? In other words, for every $(i,j)\in I^2$, is the composition
\[
M(X_i)\{r_i^-\}\stackrel{(\gamma^-_i)^t}{\longrightarrow} M^c(X)\simeq M(X)\stackrel{\gamma^+_j}{\longrightarrow}M(X_j)\{c_j^+\}
\]
zero if $i\neq j$ and the identity if $i=j$ (noting the equality $r_i^-=c_i^+$)?
\end{question}

The motivic Bia{\l}ynicki-Birula decomposition is not functorial with respect to all equivariant maps, only those that are transverse in the following sense.

\begin{defn}\label{defn transverse BB-map}
Let $f : X \ra Y$ be a $\GG_m$-equivariant morphism of semi-projective varieties with associated Bia{\l}ynicki-Birula decompositions $X = \sqcup_{i \in I} X_i^+$ and $ Y = \sqcup_{j \in J} Y_j^+$. We say $f$ is a transverse BB-map if there is an injection $\phi : I \ra J$ compatible with choices of orderings of these index sets given by Theorem \ref{thm_bb_dec} \ref{order} such that $f(X_i) \subset Y_{\phi(i)}$ and for all $i \in I$, the morphism of closed pairs $f : (X_i^+,\sqcup_{l \leq i} X_{l}^+) \ra (Y_{\phi(i)}^+,\sqcup_{l \leq i} Y_{\phi(l)}^+)$ induces a cartesian square and these closed pairs have the same codimensions.
\end{defn}

\begin{prop}\label{prop transverse BB-map}
Let $f : X \ra Y$ be a transverse BB-map as above. Then the morphism $M(f) : M(X) \ra M(Y)$ is compatible with the decompositions given by Theorem \ref{thm_bb_mot} \ref{mot BB} in the sense that we have a commutative diagram
\[\xymatrix{ M(X) \ar[r]^{M(f)}  \ar[d]^{\simeq} & M(Y) \ar[d]^{\simeq}\\ 
 \bigoplus_{i \in I} M(X_i)\{c_i^+\} \ar[r] & \bigoplus_{j \in J} M(Y_j)\{c_j^+\}
      }
  \] 
where the lower map is induced by the morphism $f$ restricted to the fixed loci.
\end{prop}
\begin{proof} 
By assumption, the morphism of closed pairs $f : (X_i^+,\sqcup_{l \leq i} X_{l}^+) \ra (Y_{\phi(i)}^+,\sqcup_{l \leq i} Y_{\phi(l)}^+)$ is transversal in the sense of \cite[Definition 1.1]{Deglise_Gysin_1} and thus the associated Gysin maps are compatible (in the sense that there is an induced commutative diagram) by \cite[Proposition 1.19 (1)]{Deglise_Gysin_1}. The claim then follows by going through the proof of Theorem \ref{thm_bb_mot} \ref{mot BB}, as the motivic BB decompositions are built from these Gysin maps.
\end{proof}

\section{Motives of equivariant semi-projective specialisations}
\label{sec:motiv-equiv-semipr}

As an application of the motivic Bia{\l}ynicki-Birula decomposition of Appendix \ref{sec mot BB}, we study $\GG_{m}$-equivariant specialisations of smooth semi-projective varieties. This is applied in the body of the article to compare the motives of the Higgs and de Rham moduli spaces on a curve, see $\S$\ref{sec motivic_NAH}. We also give an application to algebraic symplectic reductions following \cite{CBVdB, Hausel_RV_semiproj}.

Various cohomological incarnations of the following result for coarser invariants than Voevodsky motives have already appeared in the literature before, but as far as we know the conclusion is new even for Chow groups. Nakajima's proof in the appendix to \cite{CBVdB} shows the fibres of a family $f : X \ra \AA^1$ as below have the same number of rational points over a finite field $\FF_q$. For $k = \CC$ and a family $f : X \ra \AA^{1}$ of hyperk\"{a}hler reductions of a cotangent bundle to a complex vector space, there are specific instances of the topological triviality of $f$ appearing for quivers in \cite[Lemma 2.3.3]{CBVdB} and Higgs bundles in \cite[Lemma 6.1]{HT}. For a family $f : X \ra \AA^1$ over $k = \CC$ as below, Hausel, Letellier and Rogriguez-Villegas proved the fibres have isomorphic cohomology supporting pure mixed Hodge structures \cite[Theorem 7.2.1]{HLRV}

\begin{thm}\label{thm:semiproj-spe}
Let $R$ be a ring such that the exponential characteristic of $k$ is invertible in $R$. Let $X$ be a smooth quasi-projective $k$-variety equipped with a semi-projective $\GG_{m}$-action and let $f: X\to \AA^{1}$ be a smooth morphism, which is $\GG_{m}$-equivariant with respect to the given action on $X$ and a $\GG_{m}$-action on $\AA^{1}$ of positive weight. For $t\in \AA^{1}(k)$, write $X_{t}:=f^{-1}(t)$. Then, for any $t\in k$, the fibre inclusion $\iota_t : X_t \hookrightarrow X$ induces an isomorphism
  \[
M(X_{t})\simeq M(X) \in \DM(k,R).
\]
\end{thm}

\begin{proof}
We can assume that $f$ is equivariant with respect to the standard $\GG_m$-action of weight 1 on $\AA^1$ by performing a base change via a morphism $\AA^1 \ra \AA^1$ of the form $z \mapsto z^n$.

 We start by comparing the motives $M(X_{0})$ and $M(X)$. First, we observe that the $\GG_{m}$-action on $X$ restricts to $X_{0}$ and is semi-projective there. Indeed, $X_{0}$ is smooth (since $f$ is smooth) and quasi-projective, we have $X_{0}^{\GG_{m}}=X^{\GG_{m}}$, and the condition on the existence of limits is inherited from $X$. Moreover, the inclusion morphism $\iota_0: X_{0}\to X$ is a transverse BB-morphism in the sense of Definition \ref{defn transverse BB-map}; the cartesian property follows as $\iota_0$ is a closed $\GG_m$-invariant immersion and the codimension calculation is essentially performed in Nakajima's appendix in \cite{CBVdB}. By Proposition \ref{prop transverse BB-map}, the morphism $M(X_{0})\to M(X)$ is compatible with the motivic BB decompositions provided by Theorem \ref{thm_bb_mot} \ref{mot BB} and is induced by the corresponding maps on fixed points, which are isomorphisms since $X_{0}^{\GG_{m}}=X^{\GG_{m}}$. This shows that $M(X_{0})\simeq M(X)$.

We now turn to the non-zero fibres. The morphism
  \[
\sigma : X_{1}\times \GG_{m}\to X\setminus X_{0},\ (x,t)\mapsto t\cdot x
    \]
is an isomorphism because of the equivariance of $f$, so that it is enough to treat the case of $X_{1}$. For the codimension $1$ inclusion $\iota_0: X_{0}\hookrightarrow X$ we have a Gysin triangle
\[
M(X \setminus X_0) \to M(X) \to M(X_{0})\{1\}\rap.
  \]
  As with any Gysin triangle, the composition $M(X_{0})\simeq M(X)\to M(X_{0})\{1\}$ is the Euler class of the normal bundle of $X_{0}$ in $X$, which is trivial since $X_0$ is a fibre of a smooth morphism to a smooth variety. We deduce that the Gysin triangle splits. This suggests we should compare this triangle to the trivially split one coming from the pair $(X\times\AA^{1},X)$. Indeed, there is a cartesian square of closed immersions
  \[
    \xymatrix{
X_{0} \ar@{^{(}->}[d]^{\iota_0} \ar@{^{(}->}[r]^{\iota_0} & X \ar[d]^{\id\times f} \\
X \ar[r]_{\id\times \{0\}} & X\times\AA^{1} .    
    }
  \]
By \cite[Theorem 4.32]{Deglise_Gysin_II}, we deduce that there is a morphism of (shifts of) Gysin triangles
  \[
    \xymatrix{
      M(X_{0})(1)[1] \ar[r] \ar[d]_{\simeq} & M(X\setminus X_{0}) \ar[r] \ar[d] & M(X) \ar[d]_{\simeq} \ar[r]^{0} & M(X_{0})(1)[2] \ar[d]_{\simeq} \\
      M(X)(1)[1] \ar[r] & M(X\times \GG_{m}) \ar[r] & M(X\times\AA^{1}) \ar[r]^{0} & M(X)(1)[2]
      }
    \]
    from which we conclude that the morphism $(\id\times f)_{*}:M(X\setminus X_{0})\to M(X\times\GG_{m})$ is an isomorphism. Now, consider the commutative diagram
    \[
      \xymatrix{
        X_{1}\times \GG_{m} \ar[r]^{\simeq}_{\sigma} \ar[d]_{\iota_1\times \id} & X\setminus X_{0} \ar[d]_{\id\times f}  \\
        X\times\GG_{m} \ar[r]^{\simeq} & X\times\GG_{m}
      }
      \]
of $k$-varieties, where the bottom isomorphism is $(x,t)\mapsto (t\cdot x,t)$. We have shown that the right vertical map induces an isomorphism of motives, and thus we conclude that the left vertical map $\iota_1\times \id$ does as well. By K\"{u}nneth, it then follows that $\iota_1$ induces an isomorphism of motives, which concludes the proof.
\end{proof}

\begin{cor}\label{cor:semiproj-inv}
With the notations and assumptions of Theorem \ref{thm:semiproj-spe}, the morphism $X_{t}\to X$ induces an isomorphism of Chow rings with $R$-coefficients, of $\mathbb{Z}_{\ell}$-adic cohomology groups for any prime $\ell$ invertible in $k$, and more generally for any cohomology theory representable in $\DM(k,R)$ in a suitable sense.
\end{cor}
\begin{proof}
The result for Chow groups follows from the representability of Chow groups as morphism groups in $\DM(k,R)$, together with the fact that the isomorphism $M(X_{t})\simeq M(X)$ is induced by a morphism of smooth varieties, hence compatible with the diagonal and hence with the cup-product on Chow groups. For $\ell$-adic cohomology, it follows from applying the $\ell$-adic realisation functor of \cite{Ivorra-adic}.
\end{proof}  

\subsection{Applications to families of algebraic symplectic reductions}

Let $\rho : G \ra \GL(\VV)$ be a linear action of a reductive group $G$ on a finite dimensional $k$-vector space $\VV$. This induces an action of $G$ on the cotangent bundle $T^*\VV \cong \VV \times \VV^*$, which preserves the Liouville algebraic symplectic form. Let $\fg := \mathrm{Lie}\: G$ denote the Lie algebra of $G$. Then there is a moment map $\mu : T^*\VV \ra \fg^*$ given by
\[ \langle \mu(v,w), A\rangle = \langle \rho(A)v,w \rangle\]
which is an algebraic morphism that is $G$-equivariant with respect to the coadjoint action on $\fg^*$. One constructs an algebraic symplectic reduction at $0 \in \fg^*$ with respect to a character $\chi : G \ra \GG_m$ by taking a GIT quotient
\[ X_0:=\mu^{-1}(0)/\!/_{\chi} G; \]
for further details, see \cite{ginzburg} or \cite[$\S$1.1.1]{Hausel_RV_semiproj}. We can fit this into a family over $\AA^1$ by taking a 1-dimensional vector subspace $L = k \theta \subset \fg^*$ spanned by a coadjoint fixed point $\theta$ and then considering the family
\[ f: X := \mu^{-1}(L)/\!/_{\chi} G \ra \AA^1. \]
By construction of the GIT quotient, both $X$ and $X_0$ are quasi-projective varieties. We suppose that $f$ is a smooth morphism; this will be the case in certain examples for generic choices of $\chi$. 

There is a $\GG_m$-action on $X$ such that $X_0$ is an equivariant semi-projective specialisation of $X$ as described by Hausel and Rodriquez-Villegas \cite[$\S$1.1.1]{Hausel_RV_semiproj}: the dilation action of $\GG_m$ on $T^*\VV$ commutes with the $G$-action and the moment map is equivariant with respect to this action and the weight 2 action of $\GG_m$ on $\fg^*$. This  $\GG_m$-action descends to $X$ and $X_0$ such that $f$ is equivariant with respect to the $\GG_m$-action on $\AA^1$ of weight 2. Furthermore, as the $\GG_m$-action on $T^*\VV$ is semi-projective, it follows that the $\GG_m$-action on $X_0$ and $X$ are semi-projective, as both are projective over their associated affine GIT quotients; see \cite[$\S$1.1.1]{Hausel_RV_semiproj} for details. Therefore, we can apply Theorem \ref{thm:semiproj-spe} to obtain the following corollary.

\begin{cor} \label{cor:algsymp_semiproj-spe}
In the above set-up, suppose that the morphism $f : X \ra \AA^1$ is smooth. Then the fibre inclusions induce isomorphisms $M(X_t) \simeq M(X)$ in  $\DM(k,R)$ for all $t \in k$.
\end{cor}

\begin{ex}
In \cite{CBVdB} this set-up arises in order to study the Kac polynomials of quivers. Fix a quiver $Q=(V,A,h,t)$ with vertex set $V$, arrow set $A$ and head and tail maps $h,t : A \ra V$ giving the directions of the arrows. For a dimension vector $d = (d_v)_{v \in V}$, we set 
\[ \VV = \mathrm{Rep}_d(Q) := \bigoplus_{a \in A} \Hom(k^{d_{t(a)}},k^{d_{h(a)}}) \] 
and $G = \prod_{v \in V} \GL_{d_v}$, which acts linearly on $\VV$ by conjugation. Then $T^*\VV \cong \mathrm{Rep}_d(\overline{Q})$ is the representation space of the associated doubled quiver $\overline{Q}$ obtained by adding an opposite arrow $a^* : j \ra i$ for each arrow $a : i \ra j$ in $A$. One takes a generic stability parameter $\theta \in \ZZ^V$ which induces a character $\chi_\theta : G \ra \GG_m$ (see \cite{CBVdB} for details), so that for $L = k \theta$ the associated morphism $f: X \ra \AA^1$ is smooth. In this case the zero fibre $X_0$ is a moduli space of $\theta$-semistable $d$-dimensional representations of the double quiver $\overline{Q}$ satisfying the relations $\cR_0$ imposed by the zero level set of the moment map (representations of $(\overline{Q},\cR_0)$ are modules over the \emph{preprojective algebra}, see \cite{CBVdB}). By Corollary \ref{cor:algsymp_semiproj-spe}, we have
\[ M(X_0) \simeq M(X) \simeq M(X_1), \]
which lifts the result that the compactly supported $\ell$-adic cohomology of the fibres of $f$ are isomorphic in large characteristic \cite[Corollary 3.2.3]{CBVdB}.
\end{ex}

\bibliographystyle{abbrv}
\bibliography{references}

\medskip \medskip

\noindent{Freie Universit\"{a}t Berlin, Arnimallee 3, 14195 Berlin, Germany} 

\medskip \noindent{\texttt{hoskins@math.fu-berlin.de, simon.pepin.lehalleur@gmail.com}}

\end{document}